\documentclass[a4paper,leqno]{amsart}
\usepackage{amsfonts,amssymb,amscd,amsmath,bbm,amsthm,amscd}
\usepackage{accents,txfonts,url}

\newtheorem{Thm}{Theorem}[section]
\newtheorem{Lem}[Thm]{Lemma}
\newtheorem{ruleofthumb}[Thm]{Rule of Thumb}
\newtheorem{inductionlemma}[Thm]{Induction Lemma/Definition}
\newtheorem{Cor}[Thm]{Corollary}

\theoremstyle{definition}
\newtheorem{Def}[Thm]{Definition}
\newtheorem{Asm}[Thm]{Assumption}
\newtheorem{Exm}[Thm]{Example}
\newtheorem{Fact}[Thm]{Fact}
\newtheorem{Facts}[Thm]{Facts}

\theoremstyle{remark}

\newtheorem{Rem}[Thm]{Remark}

\numberwithin{equation}{section}

\newcommand{\IFF}{\ensuremath{\,\leftrightarrow\,} }
\newcommand{\DEFEQ}{\coloneqq}

\DeclareMathOperator{\labeledcoll}{labeledcoll}              
\DeclareMathOperator{\uncoll}{uncoll}              
\DeclareMathOperator{\cf}{cf}              
\DeclareMathOperator{\gen}{gen}              
\DeclareMathOperator{\dom}{dom}               
\DeclareMathOperator{\val}{val}
\DeclareMathOperator{\nor}{nor}
\DeclareMathOperator{\ro}{ro}
\DeclareMathOperator{\pos}{pos}
\DeclareMathOperator{\half}{half}
\DeclareMathOperator{\rk}{rk}                 

\newcommand{\pow}{\mathcal {P}}              
\newcommand{\card}[1]{\rvert#1\lvert}     

\newcommand{\cL}{\mathcal{L}} 
\newcommand{\al}[1]{\aleph_{#1}}          
\newcommand{\om}[1]{\omega_{#1}}         
\newcommand{\ho}{^{\omega}}                  
\newcommand{\ON}{\textrm{ON}}                  

\newcommand{\bS}[2]{\underaccent{\tilde}{\mathbf\Sigma}^#1_#2}
\newcommand{\lS}[2]{{\Sigma}^#1_#2}

\DeclareMathOperator{\trclos}{trans-clos}
\DeclareMathOperator{\ordclos}{ord-clos}
\DeclareMathOperator{\ordcol}{ord-col}
\DeclareMathOperator{\fclosure}{f-clos}

\DeclareMathOperator{\hco}{hco}

\newcommand{\mpar}{\mathfrak{p}}  

\newcommand{\inQ}{\varphi_{\in Q}}
\newcommand{\leqQ}{\varphi_{\leq Q}}

\newcommand{\ZFC}{\textrm{ZFC}}              
\newcommand{\ZFCx}{\textrm{ZFC}^\ast}        

\newcommand{\esm}{\prec}
\newcommand{\qb}{\textrm{``}}
\newcommand{\qe}{\textrm{''}}

\newcommand{\std}[1]{\check{#1}}
\newcommand{\forc}{\Vdash}
\newcommand{\incomp}{\perp}
\newcommand{\comp}{\parallel}

\newcommand{\n}[1]{\underaccent{\tilde}{#1}}
\newcommand{\ntau}{{\n{\tau}}}
\newcommand{\nsigma}{{\n{\sigma}}}

\begin{document}

\newcounter{tmp}

\date{\today}

\title{Non elementary proper forcing}

\author{Jakob Kellner}

\address{Kurt G\"odel Research Center for Mathematical Logic\\
 Universit\"at Wien\\
 W\"ahringer Stra\ss e 25\\
 1090 Wien, Austria}
\email{kellner@fsmat.at}
\urladdr{http://www.logic.univie.ac.at/$\sim$kellner}

\begin{abstract}
  We introduce a simplified framework for ord-transitive models and
  Shelah's non elementary proper (nep)
  theory.  We also introduce a new construction for the countable
  support nep iteration.
\end{abstract}

\maketitle

\section*{Introduction}

In this paper, we introduce a simplified, self contained framework for forcing
with ord-transitive models and for non elementary proper (nep) forcing, and we
provide a new construction for the countable support nep-iteration. 

Judah and Shelah~\cite{MR973109} introduced the notion ``Suslin proper'': A
forcing notion $Q\subseteq \omega^\omega$ is Suslin proper if 
\begin{itemize}
  \item ``$p\in Q$'', ``$q\leq
p$'' and ``$q\incomp p$'' (i.e., $p$ and $q$ are incompatible) are all $\bS11$
statements (in some real parameter $r$), and if 
  \item for all contable
transitive models $M$ (of some
ZFC$^*$, a sufficiently large fragment of ZFC) that contain the parameter 
$r$ and for all $p\in Q^M:=Q\cap M$ there is a $q\leq p$ which is $M$-generic,
i.e., forces that the generic filter $G$ meets every maximal antichain $A\in M$ of $Q^M$.
\end{itemize}

We always assume that $H(\chi)$ satisfies ZFC$^*$ (for sufficiently large
regular cardinals $\chi$). Then every Suslin proper forcing $Q$ is proper.
(Given an elementary submodel $N$ of $H(\chi)$, apply the Suslin proper
property to the transitive collapse of $N$.) So Suslin proper is a
generalization of properness for nicely definable forcings.

Shelah~\cite{MR2115943} introduced a generalization of Suslin proper which he
called \emph{\bf non elementary proper (nep)}.  Actually, it is a
generalization in two ``dimensions'':

\begin{enumerate}
  \item[(a)]
    We do not require ``$p\in Q$'' etc. to be defined by $\bS11$ statements, but
    rather by some arbitrary formulas that happen to be sufficiently (upwards)
    absolute. 
  \item[(b)]
    We do not require $M$ to be a transitive model, but rather a so-called
    ord-transitive model (and we allow more general parameters $r$).
\end{enumerate}

Why is (b) useful? To ``approximate'' a forcing notion $Q$ by forcings $Q^M\in
M$, it is necessary that $Q$ is the union of $Q^M$ for all possible models $M$.
(This is of course the case if $Q$ is Suslin proper: any $p\in Q$ is a real,
and therefore element of some countable transitive $M$ and thus of $Q^M=Q\cap
M$.) So if we allow only countable transitive models $M$, we can only talk
about forcings $Q$ that are subsets of $H(\al1)$. Of course there are many
other interesting forcing notions, such as iterations of length $\ge \om2$,
products of size $\ge\al2$, creature forcing constructions etc. Adopting
ord-transitive models allows us to deal with some of these forcings as well.

The motivation for (a) is straightforward: This way, we can include forcing
notions that are not Suslin proper (such as Sacks forcing), while we can still
prove many of the results that hold for Suslin forcing notions.

To summarize:
\begin{itemize}
  \item Just as Suslin proper, nep has consequences that are not 
    satisfied by all proper forcing notions. So when we know that
    a forcing is nep and not just proper, we know more about its behavior.
    And while nep implies all of the useful consequences of Suslin proper,
    nep is more general (i.e., weaker): Some popular forcings are
    nep, but not Suslin proper (e.g., Sacks forcing).
    \\
    For example, let us say that ``$Q$ preserves non-meager'' 
    if $Q$ forces that the ground model
    reals ar not meager (and analogously we define ``$Q$ 
    preserves non-Lebesgue-null).
    Goldstern and Shelah~\cite[XVIII.3.11]{MR1623206}
    proved that the proper countable support iteration $(P_\alpha,Q_\alpha)$
    of non-meager preserving forcing notions preserves non-meager,
    provided that all $Q_\alpha$ are Suslin proper.
    \\
    Shelah and the author~\cite[9.4]{MR2155272} proved that the same 
    preservation theorem holds
    for Lebesgue-null instead of meager 
    and that it is sufficient to assume (nicely
    definable) nep instead of Suslin proper.
    This has been applied by Roslanowski and
    Shelah in~\cite{MR2214118}, which proves that 
    consistently every real function is continuous on a 
    set of positive outer Lebesgue measure.  
    (The preservation theorem is applied to a forcing that 
    is nep but not Suslin proper.)
  \item In particular, forcings that are not subsets of
    $H(\al1)$ can be nep; for example big countable support 
    products. In particular, we get a preservation theorem:
    under suitable assumptions,
    the countable support iteration of nep forcings is nep.
    \\
    An example of how this can be used is
    Lemma~\ref{lem:iterationapplication} of this paper.
    (This fact was used in~\cite[4.5]{MR2353856} to investigate
    Abelian groups).
\end{itemize}

Note that the Ord-transitive models mentioned in~(b) above can be useful in a
different (and simpler) setting as well: Instead of considering a forcing
definition and the realizations $Q^V$, $Q^M$ of this definition (in $V$ or a
countable model $M$, respectively), we can just use two arbitrary (and entirely
different) forcings $Q^V\in V$ and  $Q^M\in M$ and require that $Q^M$ is an
$M$-complete subforcing of $Q^V$. In the transitive case this concept has been
a central ingredient of Shelah's oracle-cc~\cite[IV]{MR1623206}, and  it can be
applied to ord-transitive models as well. An example for such an application is
the paper~\cite{BCDBC} by Goldstern, Shelah, Wohofsky and the author, which
proves the consistency of the Borel Conjecture plus the dual Borel Conjecture.
For this construction, nep forcing is not required, just ord-transitive models.
We very briefly comment on this in Section~\ref{ss:oracle}.

\subsection*{Contents}
 
\begin{list}{}{\setlength{\leftmargin}{0.5cm}\addtolength{\leftmargin}{\labelwidth}}
  \item[Section~\ref{sec:nontransitive}, p. \pageref{sec:nontransitive}:]
    We define {\bf ord-transitive} $\epsilon$-models $M$ 
    and their forcing extensions $M[G]$.
  \item[Section~\ref{sec:nep}, p. \pageref{sec:nep}:]
    We define the notion of {\bf non elementary proper forcing}:
    $Q$ is nep, if it is nicely definable and there are generic 
    conditions for all countable models. If $Q\subseteq 2^\omega$,
    then it is enough to consider transitive models; otherwise
    models such as in Section~\ref{sec:nontransitive} are used.
  \item[Section~\ref{sec:examples}, p. \pageref{sec:examples}:]
    We mention some {\bf examples}. Rule of thumb: every nicely definable 
    forcing that can be shown to be proper is actually nep.
    We also give a very partial counterexample to this rule of thumb.
  \item[Section~\ref{sec:csi}, p. \pageref{sec:csi}:]
    We define (a simplified version of) the
    {\bf countable support iteration}
    of nep forcings (such that the limit is again nep).
\end{list}

Most of the notion and results in this paper are due to Shelah, and (most
likely) can be found in~\cite{MR2115943}, some of them explicitly (and
sometimes in a more general setting), some at least ``in spirit''. However, the
notation and many technical details are different: In many cases the notation
here is radically simplified, in other cases the notions are just incomparable
(for example the definition of nep-parameter). Most importantly, we work in
standard set theory, not in a set theory with ordinals as urelements. The
result of Subsection~\ref{ss:zap} is due to Zapletal.

\section{Forcing with ord-transitive models}\label{sec:nontransitive}

Whenever we use the notation $N\esm H(\chi)$, we imply that $N$ is countable,
and that $\chi$ is a sufficiently large regular cardinal.  We write $H(\chi)$
for the sets that are hereditarily smaller than $\chi$ and $R_\alpha$ for the
sets of rank less than $\alpha$. (We will use the notation $V_\alpha$ for
forcings extension of $P_\alpha$, the $\alpha$-th stage of some forcing
iteration.)

\subsection{Ord-transitive models}

Let $M$ be a countable set such that $(M,\in)$ satisfies  $\ZFCx$, a subset of
ZFC.\footnote{We assume that $\ZFCx$ contains a sufficient part of ZFC, in
particular extensionality, pairing, product, set-difference, emptyset, infinity
and the existence of $\om1$.}
We do not require $M$ to be transitive or elementary.
$\ON$ denotes the class of ordinals.
We use 
$\ON^M$ to denote the set of $x\in M$ such that $M$ thinks that
$x\in\ON$; similarly for other definable classes.
This notation can formally be inconsistent with the following 
notation (but as usual we
assume that the reader knows which variant is used):
\footnote{If $M$ is not transitive, then for example
the set $x=\{\alpha\in M:\, M\vDash \alpha\in\om1\}$
will generally be different from the element $y\in M$ such that $M\vDash
y=\in\om1$. In that case $x\notin M$.}
For a definable set such as $\om1$, we use $\om1^M$ to denote the element $x$
of $M$ such that $M$ thinks that $x$ satisfies the according definition.

\begin{Def}\label{def:fintransetc}
  \begin{itemize}
    \item $M$ is ord-absolute, if $\omega^M=\omega$, $\omega\subseteq M$, and
        $\ON^M\subseteq \ON$ (and therefore $\ON^M=M\cap \ON$).
    \item $M$ is ord-transitive, if it is ord-absolute 
        and $x\in M\setminus \ON$ implies $x\subset M$.
  \end{itemize}
\end{Def}


An elementary submodel $N\esm H(\chi)$ is not ord-transitive. The simplest
example of an ord-transitive model that is not
transitive is the ord-collapse of an elementary submodel:

\begin{Def}\label{def:ordcol}
  Define $\ordcol^M: M\to V$ as the transitive collapse of
    $M$ fixing the ordinals:
    \[ 
       \ordcol^M(x)= 
       \begin{cases}
        x                             & \text{if } x\in\ON \\
        \{\ordcol^M(t):t\in x\cap M\} & \text{otherwise.} 
       \end{cases}
    \]
      $\ordcol(M)\DEFEQ \{\ordcol^M(x):\, x\in M\}.$
\end{Def}

By induction one can easily show:
\begin{Fact}\label{lem:ordcol}
  Assume that $M$ is ord-absolute and 
  set $i\DEFEQ \ordcol^M$, $M'\DEFEQ\ordcol(M)$. Then
  \begin{itemize}
    \item $i:M\to M'$ is an $\in$-isomorphism.
    \item $i(x)\in\ON\IFF x\in\ON$. In particular, $M\cap \ON=M'\cap \ON$.
    \item $M'$ is ord-transitive.
    \item $i$ is the identity iff $M$ is ord-transitive.
    \item The ord-collapse ``commutes'' with the transitive collapse,
      i.e., the transitive collapse of the ord-collapse of $M$ is 
      the same as the transitive collapse of $M$.
  \end{itemize}
\end{Fact}

So if $N\esm H(\chi)$ and $H(\chi)\vDash \ZFCx$, then $M=\ordcol(N)$ is an
ord-transitive model.
This example demonstrates
that several simple formulas (that are absolute for transitive models), such as
``$x\subset z$'', ``$x\cup y=z$'' and ``$x\cap y=z$'', are {\em not} absolute
for the ord-transitive models.\footnote{%
  ``$\varphi(\bar x)$ is absolute''  means 
  $M\vDash \varphi(\bar m)$
  iff $V\vDash\varphi(\bar m)$ for all $\bar m$ from $M$.
  Let $i$ be the ord-collapse from an elementary submodel $N$ to $M$.
  Set $x=\om1$, $y=\{\{0\}\}$ and $z=x\cup y$. Then $x\in \ON$
  and $z\notin \ON$, so $i(x)=x$ and $i(z)$ is countable.
  Therefore $i(x)\cup i(y)\neq i(z)$, and $i(x)\nsubseteq
  i(z)$. Also, $i(z)\cap i(x)\neq i(x)$.%
}
However, a few simple properties are absolute: In particular, if a formula
$\varphi(r)$ about real numbers is absolute for all transitive models, then is
absolute for all ord-transitive models as well (which can easily be seen using
the transitive collaps, cf.~the following Fact~\ref{fact:labeled}). We 
now mention 
some of these absolute properties for ord-transitive models $M$:
  \begin{itemize}
    \item $x\in \omega^\omega$ is absolute; every $\mathbf\Sigma^1_1$ formula
      is absolute; 
    \item ``Finite sets'' are absolute: 
        $z=\{x,y\}$ is absolute, if $x\in M$ and $x$ is finite,
        then $x\subset M$ and $M\vDash\qb x\text{ is finite}\qe$.
        $H^M(\al0) = H(\al0)$.
    \item If $M\vDash f:A\to B$, then $f:{A\cap M}\to{B\cap M}$.
       If additionally $M$ thinks that $f$ is injective (or surjective), then
       $f$ is injective (or surjective with respect to the
       new image).
    \item $x\in R_\alpha $ is upwards absolute. If additionally
      $x\notin \ON$, then $\card{x}\leq \card{\alpha}$ is upwards
      absolute.
    \item If either $x \in\ON$, or
        $x\cap \ON=\emptyset$, then $y\subset x$ is absolute.
    \end{itemize}

Instead of ord-transitive models, we could equivalently use transitive models
with an (ordinal) labeling on the ordinals:

\begin{Def}\label{def:labeledmodel}
  A labeled model is a pair $(M,f)$ consisting 
  of a transitive, countable $\ZFCx$ model $M$
  and a strictly monotonic function $f:(M\cap\ON)\to\ON$ satisfying
  $f(\alpha)=\alpha$ for $\alpha\leq \omega$.
  \\
  Given a labeled model $(M,f)$, define a map $i:M\to V$ by
  \[i(x)=\begin{cases}
    f(x)&\text{if }x\in\ON\\
    \{i(y):\, y\in x\} &\text{otherwise.}
  \end{cases}\]
  Set $\uncoll(M,f):=i[M]$.
  \\
  Given an ord-transitive model $M$, let $j:M\to M'$ be the transitive
  collapse (an $\in$-isomorphism) and
  let $f:M'\cap\ON\to\ON$ be the inverse of
  $j$. Define $\labeledcoll(M):=(M',f)$.
\end{Def}

By induction, one can prove the following:
\begin{Fact}\label{fact:labeled}
  If $M$ is an ord-transitive model, then $\labeledcoll(M)$ is a labeled
  model and $\uncoll(\labeledcoll(M))=M$.
  If $(M,f)$ is a labeled model, then $\uncoll(M,f)$ is an ord-transitive
  model and $\labeledcoll(\uncoll(M,f))=(M,f)$.
\end{Fact}
We say that the ord-transitive model $M$ and the labeled model $(M',f')$
correspond to each other, if $M=\uncoll(M',f')$ or equivalently
$(M',f')=\labeledcoll(M)$. So each ord-transitive model corresponds to exactly one
labeled model and vice versa.

This also shows that is easy to create ``weird'' ord-transitive models; in
particular ``$\alpha$ is successor ordinal'' and similarly simple formulas are
generally not absolute for ord-transitive models.
We will generally not be interested in such weird models:

\begin{Def} Let $M$ be ord-transitive.
  \begin{itemize}
    \item $M$ is ``successor-absolute'', if ``$\alpha$ is
      successor'' and ``$\alpha=\beta+1$'' both are absolute between $M$ and $V$.
    \item A successor-absolute $M$ is $\cf\omega$-absolute, 
      if ``$\cf(\alpha)=\omega$''
      and  ``$A$ is a countable cofinal subset of $\alpha$'' both are
      absolute between $M$ and $V$.
  \end{itemize}
\end{Def}

\begin{Fact}
If $M$ is $\cf\omega$-absolute and $M$ thinks that $x$ is countable, then
$x\subset M$.
\end{Fact}

\begin{proof}
  If $x\notin \ON$, then $x\subseteq M$.  So assume towards a contradiction
  that $x\in\ON$ is minimal with $x\nsubset M$ (and $x<\om1^M$).  $M$ thinks
  that $y:=x\setminus \{0\}$ (constructed in $M$) is countable and cofinal in
  $x$.  Since $y\notin \ON$ we know $y\subset M$, so $x=\bigcup_{\alpha\in
  y}\alpha$ is a subset of $M$, since $x$ was the minimal counterexample.)
\end{proof}

$M$ is successor-absolute iff the corresponding labeled model $(M',f')$
satisfies: $f(\alpha+1)=f(\alpha)+1$ and $f(\delta)$ is a limit ordinal for all
limit ordinals $\delta$.

\begin{Rem}
\begin{itemize}
\item
We will see in the  next section how to construct forcing extensions for
ord-transitive models $M$, or equivalently labeled models $(M',f')$: 
If $G$ is $M$-generic, and $G'$ the image under the transitive collapse
(which will be $M'$-generic), then the forcing extension $M[G]$ is just
the ord-transitive model corresponding to $(M'[G'],f')$.
Such forcing extensions are the most important ``source'' for ord-transitive
models that are not just (the ord-transitive collapse of) an elementary model. 
\item
In applications, we typically have to deal with ord-transitive models that are
internal forcing extensions of elementary models (i.e., in the construction
above $G$ is in $V$ and $M$ is the ord-collapse of $N\esm H(\chi)$).
\item
All such ord-transitive models are ord-absolute (and satisfy many additional
absoluteness properties).  So for applications, it is enough to only consider
ord-absolute models, and restrictions of this kind sometimes make notation
easier.
\item
Ord-collapses $M$ of elementary submodels are $\cf\omega$-absolute.  The same
holds for forcing extensions $M[G]$ by proper forcing notions. However, general
internal forcing extensions $M[G]$ will not be $\cf\omega$-absolute: If $G$ is
generic for a collapse forcing, then $M[G]$ will think that $\om1^V$ is
countable.  In some applications (such as the the preservation theorem
mentioned in the introduction) it is essential to use such collapses, therefore
we cannot restrict ourselves to $\cf\omega$-absolute models.  However, for some
other applications, $\cf\omega$-absolute models are sufficient (e.g., for the
application mentioned in Section~\ref{ss:oracle}).
\end{itemize}
\end{Rem}

Every ord-transitive model is hereditarily countable modulo ordinals:
\begin{Def}\label{def:hco}
  \begin{itemize}
    \item We define $\ordclos$ by induction:
       $\ordclos(x)=x\cup\bigcup\{\ordclos(t):\, t\in x\setminus \ON\}$.
    \item $\hco(\alpha)=\{x\in R_\alpha:\, \card{\ordclos(x)}\leq \al0\}$
    \item $\hco=\bigcup_{\alpha\in\ON}\hco(\alpha)$.
\end{itemize}
\end{Def}

For example, if $\alpha>\om1$, then $\om1$ is element of $\hco(\alpha)$, but
$\om1\cup \{\{\emptyset\}\}$ or $\om1\setminus \{\emptyset\}$ are not.

\begin{Facts}\label{lem:hc}
  \begin{itemize}
    \item $\ordclos(M)$ is the smallest ord-transitive superset of $M$.
    \item
      A ZFC$^*$-model $M$ is ord-transitive iff $\ordclos(M)=M$.
    \item If $M$ is ord-transitive and countable, then
      $M\in\hco$.
    \item If $M$ is ord-transitive and $x\in M$, then 
      $\ordclos(x)=\ordclos^M(x)\subseteq M$.
    \item ``$x\in\hco(\alpha)$'' is upwards absolute for ord-transitive
      models.
  \end{itemize}
\end{Facts}

As already mentioned, there is an ord-transitive model $M$ such that $\om1^V$
is countable in $M$.  So $M$ thinks that $\om1^V$ is not just element of $\hco$
(which is true in $V$ as well), but that it can also be constructed as
countable set (which is false in $V$).


\subsection{Forcing extensions}

Forcing still works for ord-transitive models (but the evaluation of names has
to be modified in the natural way). In the following, $M$ always denotes an
ord-transitive model. 

\begin{Def}\label{def:MQgen}
  Let $M$ think that $\leq$ is a partial order on $P$. So in $V$, $\leq$ is a
  partial order on $P\cap M$. Then $G$ is called $P$-generic over $M$ (or just
  $M$-generic, or $P$-generic), if $G \cap P\cap M$ is a filter on $P\cap M$ and meets
  every dense subset $D\in M$ of $P$.\footnote{I.e.: If $p,q\in G \cap P\cap M$,
    then there is a $r\leq p,q$ in $G \cap P\cap M$;
    and if $D\in M$ and $M$ thinks that $D$ is a dense subset of
    $P$ (or equivalently: $D\cap M$ is a dense subset of $P\cap M$)
    then $G\cap D\cap M$ is nonempty.}
\end{Def}
To simplify notation, we will use the following assumption:
\begin{Asm}\label{asm:notation1}
  $P\cap \ON$ is empty.
  (Then in particular $P\subseteq M$, and we can write $P$ instead of $P\cap M$.
  Also, if $D\subset P$ is in $M$, then $D\subset M$.)
\end{Asm}

In Definition~\ref{def:MQgen} we do not assume $G\subseteq P$. This slightly
simplifies notation later on.  Obviously $G$ is $M$-generic iff $G \cap P$ is
$M$-generic.  One could equivalently use maximal antichains, predense sets, or
open dense sets instead of dense sets in the definition (and one can omit the
``filter'' part if one requires that a maximal antichain $A$ in $M$ meets the
filter $G$ in exactly one point).

Let $\labeledcoll(M)=(M',f')$ be the labeled model corresponding to $M$, via
the transitive collapse $j$. Let $G\subseteq P$ and set $P':=j(P)$ and
$G':=j[G]$. Since the transitive collapse is an isomorphism, $G'$ is
$P'$-generic over $M'$ iff $G$ is $P$-generic over $M$. In that case we can
form the forcing extension $M'[G']$ in the usual way, and define
$M[G]=\uncoll(M'[G'],f')$ as the ord-transitive model corresponding to
$(M'[G'],f')$. Let $J: M[G]\to M'[G']$ be the transitive collapse, and $I$ its
inverse, then we can define $\n\tau[G]^M$ as $I(J(\n\tau)[G'])$ for a $P$-name
$\n\tau$ in $M$. Elementarity shows that this is a ``reasonable'' forcing
extension.

We now describe this extension in more detail and using the ord-transitive
model $M$ more directly:

Basic forcing theory shows: If $M$ is a transitive model, $P\in M$, and $G$ a
$P$-generic filter over $M$, then we can define the evaluation of names by
\begin{equation}\label{eq:gurkx}
  \n\tau[G]=\{\n\sigma[G]:\, (\n\sigma,p)\in\n\tau,\, p\in G\},
\end{equation}
and 
$M[G]$ will be a (transitive) forcing extension of $M$. 

This evaluation of names works for elementary submodels as well, provided that
$G$ is not only $N$-, but also $V$-generic.  More exactly:
If $N\esm H(\chi)$ contains $P$, and if $G$ is $N$- {\em and} $V$-generic, then
$N[G]$ is a forcing extension of $N$ (and in particular end-extension).  Here
it is essential that $G$ is $V$-generic as well: If $N\esm H(\chi)$ and $G\in
V$ is $N$-generic (for any nontrivial forcing $P$), then $N[G]$ is not an
end-extension of $N$, since $G\in \pow(P)\in N$, but  $G\notin N$.

This can be summarized as follows:
\begin{Fact}\label{thm:transforcing}
  Assume that either $M$ is transitive and $G$ is $M$-generic, or that 
  $M\esm H(\chi)$ and $G$ is $M$- and $V$-generic. Then
  \begin{itemize}
    \item $M[G]\supset M$ is an end-extension\footnote{%
Any usual concept of forcing extension (with regard to pairs of
$\in$-models) will require that $M[G]$ is an end-extension of $M$: If $\ntau$
is forced to be in some $x$ with $x\in V$, 
then the value of $\ntau$ can be decided by a
dense set.  Similarly, we get: $M$ is $M[G]$ intersected with the transitive
closure of $M$.}
      (i.e., if $y\in M[G]$ and $y\in x\in M$, then $y\in M$),
      and $\ON^{M[G]}=\ON^M$.
    \item $M[G]\vDash \varphi(\ntau[G])$ iff $M\vDash p\forc \varphi(\ntau)$
      for some $p\in G$.
  \end{itemize}
  In the transitive case $M[G]$ is transitive; and
  in the elementary submodel case, we get:
  \begin{itemize}
    \item $(M[G],\epsilon, M)\esm (H^{V[G]}(\chi),\epsilon,H^V(\chi))$.
    \item Forcing extension commutes with transitive collapse: 
      Let $i$ be the transitive collapse of $M$, and $I$ of
      $M[G]$. Then $I$ extends $i$,
      $i[G]$ is $i[M]$ -generic and 
      $i(\ntau)[i[G]]=I(\ntau[G])$.
  \end{itemize}
\end{Fact}

If one considers general ord-transitive candidates $M$ (i.e., $M$ is neither
transitive nor an elementary submodel), then Definition~\eqref{eq:gurkx} does not
work any more.
  For example, if $M$ is countable and thinks that
$\ntau$ is a standard name for the ordinal $\om1^V$, then $\ntau\subset M$ is
countable, so $\ntau[G]$ will always be countable and different
from $\om1^V$.
This leads to the following natural modification of~\eqref{eq:gurkx}:

\begin{Def}\label{def:evaluation}
  Let $G$ be $P$-generic over $M$, and let
  $M$ think that $\ntau$ is a $P$-name.
      \[
        \ntau[G]^M\DEFEQ 
        \begin{cases} 
          x, \text{ if } x\in M\,\&\, (\exists p\in G\cap P)\,
          M\vDash\qb p\forc \ntau=\std{x}\qe
          \\
          \{ \nsigma[G]^M:\, (\exists p\in G\cap P)\, (\nsigma, p)\in
          \ntau \cap M \}\text{ otherwise.}
        \end{cases}
      \]
      $M[G]\DEFEQ \{\ntau[G]^M:\, \ntau\in M,\ M\vDash\qb\ntau\text{ is a } P\text{-name}\qe\}$.
\end{Def}
(Note that being a $P$-name is absolute.)

We usually just write $\ntau[G]$ instead of $\ntau[G]^M$.  There should be no
confusion which notion of evaluation we mean, \ref{def:evaluation} or
\eqref{eq:gurkx}, which we can also write as $\ntau[G]^V$:
\begin{itemize}
  \item If $M$ is transitive, then $\ntau[G]^M=\ntau[G]^V$.
  \item If $M$ is elementary submodel (and $G$ is $M$- and $V$-generic),
    then we use $\ntau[G]^V$. ($\ntau[G]^M$ does not lead to a
    meaningful forcing extensions.)\footnote{%
If $M$ is not ord-transitive, e.g., $M\esm H(\chi)$, then $\ntau[G]^M$ does
not lead to a meaningful forcing extension: Let $P$ be the countable partial
functions from $\om1$ to $\om1$, and  let $G$ be $M$-generic ($G$ can
additionally be $V$-generic as well).  Let $\n\Gamma\in M$ be the canonical
name for the generic filter $G$. So $\n\Gamma[G]^M$ is countable.  Since $P$
is $\sigma$-closed, $\n\Gamma[G]^M\in \pow^V(P)\in M$, so $M[G]$ (using the
modified evaluation) is not an end-extension of $V$.%
}
  \item If $M$ is ord-transitive,
    then we use $\ntau[G]^M$.
\end{itemize}
\begin{Rem}
  The omission of $M$ in $\ntau[G]^M$ should not hide the fact that for
  ord-transitive models, $\ntau[G]^M$ trivially {\em does} depend on $M$:
  If for
  example $M_1\cap\beta=\alpha<\beta$ and $M_1$ thinks that $\ntau$ is a
  standard
  name for $\beta$, and if $M_2$ contains $P$, $\ntau$ and $\alpha$, then then
  $\ntau[G]^{M_1}=\beta\neq \ntau[G]^{M_2}$. 
\end{Rem}

$\ntau[G]$ is well-defined only if $G$ is $M$-generic, or at least a filter.
(If $G$ contains $p_0\incomp_P p_1$, then there is (in $M$) a name $\ntau$ and
$x_0\neq x_1$ such that $p_i$ forces $\ntau=x_i$ for $i\in\{0,1\}$.)

If $M$ is ord-transitive then the basic forcing theorem works as usual
(using the modified evaluation):
\begin{Thm}\label{thm:forcing}
  Assume that $M$ is ord-transitive and that $G$ is $M$-generic. Then
  \begin{itemize}
    \item $M[G]$ is ord-transitive.
    \item $M[G]\supset M$ is an end-extension.
      $\ON^{M[G]}=\ON^M$.
    \item $M[G]\vDash \varphi(\ntau[G])$ iff $M\vDash p\forc \varphi(\ntau)$
      for some $p\in G\cap P$.
  \end{itemize}
  Moreover, the transitive collapse commutes with the forcing extension:
  Let $(M',f')$ correspond to $M$, and $G'$ the image of $G$ under the 
  transitive collapse. Then $(M'[G'],f')$ corresponds to $M[G]$.
\end{Thm}

(The proof is a straightforward induction.)
So forcing extensions of ord-transitive models behave just like the usual
extensions. For example, we immediately get:

\begin{Cor}\label{cor:pforc}
  If $M$ is countable and ord-transitive, then
  $M\vDash\qb p\forc \varphi(\ntau)\qe$ iff
  $M[G]\vDash\varphi(\ntau[G])$
  for every $M$-generic filter $G$ (in $V$) containing $p$.
\end{Cor}


\begin{Fact}\label{lem:evalabs}
   Assume that $N$ is ord-transitive, $M\in N$, $P\in M$.
   Then the following are absolute between $N$ and $V$
   (for $G\in N$ and $\ntau\in M$):
   \begin{itemize}
     \item $M$ is ord-transitive.
     \item $G$ is $M$-generic, and
     \item (assuming $M$ is ord-transitive and $G$ is $M$-generic) $\ntau[G]^M$.
   \end{itemize}
\end{Fact}
The last item means that we get the same value for $\ntau[G]^M$ whether we
calculate it in $N$ or $V$.  It does {\em not} mean $\ntau[G]^M=\ntau[G]^N$.
(If $\ntau$ is in $M$, then $\ntau[G]^N$ will generally not be an interesting
or meaningful object.)

Let us come back once more to the proper case. By induction on
the rank of the names we get that the ord-collapse and 
forcing extension commute:
\begin{Lem}
  Assume that $N\esm H(\chi)$, and $P\in N$. Let $i:N\to M$ be the ord-collapse.
  \begin{itemize}
    \item $G\subseteq P$ is $N$-generic iff $i[G]$ is $M$-generic.
    \item Assume that $G$ is $N$- and $V$-generic. Then the ord-collapse $I$ of
      $N[G]$ extends $i$, and $I(\ntau[G])=(i(\ntau))[i[G]]$.
    \item If $P\subseteq \hco$, then $i$ is the identity on $P$.
  \end{itemize}
\end{Lem}


\subsection{$M$-complete subforcings}\label{ss:oracle}

In the rest of the paper,  we will use ord-transitive models in the context of
definable proper forcings (similar to Suslin proper).  But first let us briefly
describe another, simpler, setting in which ord-transitive models can be used.

Let $M$ be a countable transitive model and $Q^M$ a forcing notion in $M$. We
say that $Q^M$ is an $M$-complete subforcing of $Q\in V$, if $Q^M$ is a
subforcing of $Q$ and every maximal antichain $A\in M$ of $Q^M$ is a maximal
antichain in $Q$ as well.
So there are two differences to the ``proper'' setting: $Q^M$ and $Q$
do not have to be defined by the same formula,\footnote{They do not have to be
nicely definable at all, and furthermore $Q^M$ and $Q$ can be entirely
different: E.g., $Q^M$ could be Cohen forcing in $M$ and $Q$ could be
(equivalent to) random forcing in $V$.}
and we do not just require that below every condition in $Q^M$ we find a
$Q^M$-generic condition in $Q$, but  that already the empty condition is
$Q^M$-generic.\footnote{In the proper case, this is equivalent to ``$Q$ is
ccc''.}

For transitive models, this concept has been used for a long time. It is, e.g.,
central to Shelah's oracle-cc~\cite[IV]{MR1623206}. In oracle-cc forcing, one
typically constructs a forcing notion $Q$ of size $\al1$ as follows: Construct
(by induction on $\alpha\in\om1$) an increasing (non-continuous) sequence of
countable transitive models $M^\alpha$ (we can assume that $M^\alpha$ knows
that $\alpha$ is countable), and forcing notions $Q^\alpha\in M^\alpha$ such
that $Q^\delta=\bigcup_{\beta<\delta} Q^\beta$ for limits $\delta$ and such
that $Q^{\beta+1}$ is an $M^\beta$-complete superforcing of $Q^\beta$.  (Then
each $Q^\alpha$ will be $M^\alpha$-complete  subforcing of the final $Q$.) So
we use the pair $(M^\alpha,Q^\alpha)$ as an approximation to the final forcing
notion $Q$. Since we use transitive models, this $Q$ has to be subset of
$H(\al1)$.

If we want to investigate other forcings, we can try to use ord-transitive
models instead. For example, in~\cite{BCDBC} we use a forcing iteration $\bar
P=(P_\alpha,Q_\alpha)_{\alpha\le\om2}$ (where each $Q_\alpha$ consists of
conditions in $H(\al1)$), and we ``approximate'' $\bar P$ by pairs $(M^x,\bar
P^x)$, where $M^x$ is a countable ord-transitive model and $M^x$ thinks that
$\bar P^x$ is a forcing iteration of length $\om2^V$. Instead of assuming that
$P^x_{\om2}$ is a subforcing of $P_{\om2}$, it is more natural to assume
(inductively) that each $P^x_\alpha$ can be canonically (and in particular
$M^x$-completely) embedded into $P_\alpha$, and that $P_\alpha$ forces that
$Q^x_\alpha[G^x_\alpha]$ (evaluated by the induced $P^x_\alpha$-generic filter
$[G^x_\alpha]$) is an $M^x[G^x_\alpha]$-complete subforcing of $Q_\alpha$.  We
show that given $\bar P^x$ in a countable ord-transitive model $M^x$ we can
find variants of the finite suppost and the countable support iterations $\bar
P$ such that $\bar P^x$ canonically embeds into $\bar P$ (and we show that some
preservation theorems that are known for proper countable support iterations
also hold for this variant of countable support). For this 
application it is enough to consider $\cf\omega$-absolute models.

In the current paper, we do something very similar (in the nep setting, i.e.,
the definable/proper framework), in Sections~\ref{ss:nonabs}
and~\ref{ss:subsetiteration}. Let us again stress the obvious difference: In
the nep case, we use definable forcings, and $Q^x_\alpha$ is the evaluation in
$M^x[G^x_\alpha]$ of the same formula that defines $Q_\alpha$ in $V[G_\alpha]$,
and we just get that below (the canonical image of) each $p\in P^x_\alpha$
there is some $M^x$-generic $q\in P_\alpha$. 

In particular, the application of non-wellfounded models in~\cite{BCDBC} does
not use any of the concepts that are introduced in the rest of this paper.

\section{Nep forcing}\label{sec:nep}


\subsection{Candidates}





We now turn our attention to definable forcings. More particularly, we will
require that for all suitable (ord-transitive) models $M$, ``$x\in Q$'' is
upward absolute between $M$ and $V$.%
\footnote{There are useful notions similar to nep without this property. 
  Examples for such forcings appear naturally when iterating nep forcings,
  cf. Subsection~\ref{ss:nonabs}.}
Also, we will require that for all $x\in Q$ there is a model $M$ knowing that
$x\in Q$.  This is only possible if $Q\subset \hco$ (since every countable
ord-transitive model is hereditarily countable modulo ordinals),
but it is not required that $Q\subseteq H(\al1)$ (as it is the case when using
countable transitive models only).

It is
natural to allow parameters other than just reals.  The following is a simple
example of a definable iteration using a function $\mpar: \om1\to 2$ as
parameter:
$(P_\beta,Q_\beta)_{\beta<\om1}$ is the countable support iteration such that
$Q_\beta$ is Miller forcing if $\mpar(\beta)=0$ and
random forcing if $\mpar(\beta)=1$.

Once we use such a parameter $\mpar$, we of course cannot assume that $\mpar$
is in the model $M$ (since $M$ is countable and ord-transitive). Instead, we
will assume
that $M$ contains its own version $\mpar_M$ of the parameter; in our example
we would require that $\delta:=\om1^V\in M$ and that $M$ thinks that $\mpar_M$
is a function from $\delta$ to 2, (so really $\dom(\mpar_M)=\delta\cap M$) and
we require that $\mpar_M(\beta)=\mpar(\beta)$ for all $\beta\in M$. 

More generally we define ``$\mpar$ is a nep parameter'' by induction 
on the rank: $\emptyset$ is a nep-parameter, and
\begin{Def}
  $\mpar$ is a nep-parameter, if $\mpar$ is a function
  with domain $\beta\in\ON$ and $\mpar(\alpha)$
  is a nep-parameter for all $\alpha\in\beta$.\\
  Let $M$ be an ord-transitive model. Then $\mpar_M$ is the $M$-version of
  $\mpar$, if $\dom(\mpar_M)=\dom(\mpar)\cap M$ and
  $\mpar_M(\alpha)$ is the $M$-version
  of $\mpar(\alpha)$ for all $\alpha\in \dom(\mpar_M)$.
\end{Def}

In other words: A nep-parameter is just an arbitrary set together with a
hereditary wellorder.

If $M$ contains $\mpar_M$, then $M$ thinks that $\mpar_M$ is a nep-parameter
(and if $\beta=\dom(\mpar)$, then $\beta\in M$ and $M$ thinks
$\beta=\dom(\mpar_M)$).

We can canonically code a real $r$, an ordinal, or a subset of the ordinals as
a nep-parameter.

\begin{Def}
  Let $\mpar$ be a nep-parameter.
  $M$ is a $(\ZFCx,\mpar)$-candidate, if $M$ is a countable, ord-transitive,
  successor absolute
  model of $\ZFCx$ and contains $\mpar_M$, the $M$-version of
  $\mpar$.
\end{Def}

We can require many additional absoluteness conditions for candidates, e.g.,
the absoluteness of the canonical coding of $\alpha\times\alpha$, or
$\cf\omega$-absoluteness.  The more conditions we require, the less candidates
we will get, i.e., the weaker the properness notion ``for all candidates, there
is a generic condition'' is going to be. In practice
however, these distinctions do not seem to matter: All nep
forcings will satisfy the (stronger) official definition, and for all
applications  weaker versions suffices.

To be more specific: Most applications will only use properness for candidates
$M$ that satisfy
\begin{equation}\label{eq:internalextension}
  \text{$M$ is an internal forcing extension of an elementary submodel $N$.}
\end{equation}
More exactly: We start with $N\esm H(\chi)$, pick some $P\in N$, set
$(N',P')=\ordcol(N,P)$, and let $G\in V$ be $P'$-generic over $N'$.
Some application might also use
\begin{equation}\label{eq:cccextension}
  \text{$M$ is an elementary submodel in a $P$-extension, for a $\sigma$-complete $P$.}
\end{equation}
More exactly: Let  $P$ be $\sigma$-complete, pick in the $P$-extension $V[G]$
some $N\esm H^{V[G]}(\chi)$ and let $N'$ be the ord-collapse. Then $N'$ is in
$V$ (and an ord-transitive model).

Of course all these models satisfy a variety of absoluteness properties (such
as the canonical coding of $\alpha\times\alpha$ etc).  So for all applications,
it would be enough to consider candidates that
satisfy~\eqref{eq:internalextension} (or some exotic application might
need~\eqref{eq:cccextension}), but we we do not
make the properties~\eqref{eq:internalextension} or~\eqref{eq:cccextension}
part of the official definition of ``candidate'', since both properties are
much more complicated (and less absolute) than just ``$M$ is a countable,
ord-transitive $\ZFCx$-model''.

Note however that generally we can {\em not} assume that the $P$ used
in~\eqref{eq:internalextension} is proper or even just $\om1$-preserving. For
example in the application in~\cite{MR2155272}, we need $P$ to be a collapse of
$\al1$.  So in particular we can not assume that all candidates are
$\cf\omega$-absolutene.

We will only be interested in the normal case:
\begin{Def}\label{def:zfcnormal}
  $\ZFCx$ is normal, if $H(\chi)\vDash\ZFCx$ for sufficiently large
  regular $\chi$.
\end{Def}

Sometimes we will assume that $\ZFCx$ is element of a candidate $M$. This allows
us to formulate, e.g., ``$M$ thinks that $M'$ is a candidate''.  We can
guarantee this by choosing $\ZFCx$ recursive, or by coding it into $\mpar$. 

\begin{Lem}\label{lem:basiccandidates}
  \begin{enumerate}
    \item (Assuming normality.) If $N\esm H(\chi)$ contains $\mpar$, and $(M,\mpar_M)$
      is the ord-collapse of $(N,\mpar)$, then $M$ is candidate and
      $\mpar_M$ is the $M$-version of $\mpar$.
    \item The statements ``$\mpar_M$ is the $M$-version of $\mpar$''
      is absolute between transitive universes. If
      $\mpar_M$ is the $M$-version of $\mpar$, and $M$ thinks
      that $M'$ is ord-transitive and
      that $\mpar_{M'}$ is the $M'$-version of $\mpar_M$, then
      $\mpar_{M'}$ is the $M'$-version of $\mpar$.
    \item If $M[G]$ is a forcing extension of $M$, and $\mpar_M$
      the $M$-version of $\mpar$, then $\mpar_M$ is also the
      $M[G]$-version of $\mpar$.
    \item For $x\in\hco$, a nep parameter $\mpar$ 
      and a theory $T$ in the language $\{\in, c^x, c^\mpar\}$,
      the existence of a candidate $M$ containing $x$ such that
      $(M,\in,x,\mpar_M)$ satisfies $T$ is absolute between universes
      containing $\om1^V$ (and, of course, $x$,$\mpar$ and $T$).
  \end{enumerate}
\end{Lem}

This is straightforward, apart from the last item, which follows from the
following modification of Shoenfield absoluteness.

\begin{Rem}
  Shelah's paper~\cite{MR2115943} uses another notion of nep-parameter
  With our definition, for
  every $\mpar$ and $M$ there is exactly one $M$-version $\mpar_M$ of $\mpar$,
  but this is not the case for Shelah's notion. (There, a
  candidate is defined as pair $(M,\mpar_M)$ such that $\mpar_M\in M$
  is an $M$-version of $\mpar$.) Both notions satisfy  
  Lemma~\ref{lem:basiccandidates}.
\end{Rem}

\begin{Lem}\label{lem:shoenfieldlevy}
  Assume that
  \begin{itemize}
    \item $S$ is a set of sentences in the first order language 
      using the relation symbol $\in$ and the constant symbols $c^x,c^\mpar$,
    \item $\ZFCx\subseteq \ZFC$,
    \item $L'$ is a transitive ZFC-model
      (set or class) containing 
      $\ZFCx$, $\om1^V$, $\mpar$, and $S$,
    \item $x\in \hco^{L'}$.
  \end{itemize}
  If in $V$ 
  there is a $(\ZFCx,\mpar)$-candidate $M$ containing $x$
  such that $(M,\in,x,\mpar_M)\vDash S$, then there is such a
  candidate in $L'$.
\end{Lem}

\begin{proof}
  We call such a candidate a good candidate. 
  So we have to show:
  \begin{equation}
    \text{If there is a good candidate in $V$,
    then there is one in $L'$.}
  \end{equation}
  Just as in the proof of Shoenfield absoluteness,
  we will show that 
  a good candidate $M$ corresponds to an infinite descending
  chain in a partial order $T$
  defined in $L'$. (Each node of $T$ is a finite 
  approximation to $M$). Then we use that the
  existence of such a chain is absolute.

  We define for a nep-parameter $y$
  \begin{equation}
    \fclosure(y)=\{ y(a):\, a\in\dom(y)\}\cup \bigcup_{a\in
    \dom(y)}\fclosure(y(a)).
  \end{equation}
  So every $z\in\fclosure(y)$ is again a nep-parameter.

  Fix in  $L'$ for every $y\in (\{x\}\cup \trclos(x))\setminus\ON$
  an enumeration 
  \begin{equation}\label{eq:myenum1}
    y=\{f^y(n):\, n\in\omega\}.
  \end{equation}
  Also in $L'$, we fix some $\delta\geq \om1^V$ bigger than
  every ordinal in $\{x\}\cup \trclos(x)$ and bigger than
  $\dom(y)$ for every $y\in\fclosure(\mpar)\cup\{\mpar\}$.

  We can assume that $S$ contains $\ZFCx$ as well as the sentence
  ``$c^\mpar$ is a nep-parameter''. 
  We use (in $L'$) the following fact:
  \begin{quote}
    Let $S$ be a theory of the countable (first-order) language $\cL_S$.  Then
    there is a theory $S'$ (of a countable language $\cL_{S'}\supset \cL_S$) 
    such that the
    deductive closure of $S'$ is a conservative extension of $S$, and every
    sentence in $S'$ has the form $(\forall x_1)(\forall x_2)\dots (\forall x_n)
    (\exists y)\, \psi(x_1,\dots,x_n,y)$ for some quantifier free
    formula $\psi$ (using new relation symbols of $S'$).
  \end{quote}
  
  So we fix $S'$ and $\cL'$,
  consisting of relation
  symbols $R_i$ ($i\in\omega$) of arity $r_i\geq 1$,
  and constant symbols $c_i$ ($i\in\omega$).
  We can assume that there are constant symbols 
  for $\omega$ and for each natural number.
  We can further assume
  \begin{itemize}
    \item $c_0=c^x$, $c_1=c^\mpar$, 
    \item $R_0=R^\epsilon(x,y)$ expresses $x\in y$,
    \item $R_1=R^{\dom}(x,y)$ expresses ``$x$ is a function and $\dom(x)=y$'',
    \item $R_2=R^{\fclosure}(x)$ expresses $x\in\fclosure(c^\mpar)\cup\{c^\mpar\}$,
    \item $R_3=R^{\ON}(x)$ expresses $x\in\ON$.
  \end{itemize}

  We set $\cL'_i=\{R_0\dots R_{i-1},c_0\dots c_{i-1}\}$.
  and fix an enumeration 
  $(\varphi_i)_{i\in\omega}$ of all sentences in $S'$ such that
  $\varphi_i$ is a $\cL'_i$-sentence.
  We now define the partial order $T$ as follows:
  A node $t\in T$
  consists of the natural number 
  $n^t$, the sequences
  $(c_i^t)_{i\leq n^t}$ and $(R_i^t)_{i\leq n^t}$,
  and the following functions with domain $n^t$:
  $\textrm{ord-val}^t$, $x\textrm{-val}^t$,
  $\mpar\textrm{-val}^t$, and $\textrm{rk}^t$
  such that the following is satisfied:
  \begin{itemize}
    \item $n^t\geq 4$.
       We interpret $n^t=\{0,\dots,n^t-1\}$ to be the universe of the 
       following $\cL'_{n^t}$-structure:
       $c_i^t\in n^t$ is the $t$-interpretation of $c_i$
       and $R_i^t\subseteq (n^t)^{r_i}$ is the $t$-interpretation of $R_i$
       for all $i<n^t$.
    \item $\textrm{ord-val}^t: n^t\to  \delta\cup\{\textrm{na}\}$.
          If $c_i$ is the constant symbol for some $m\leq \omega$,
          then $\textrm{ord-val}^t(c_i^t)=m$. If
          $\textrm{ord-val}^t(a)\neq \textrm{na}$, then we have the 
          following: ${R^\ON}^t(a)$
          holds, and
          ${R^\in}^t(b,a)$ holds iff
          $\textrm{ord-val}^t(b)\in \textrm{ord-val}^t(a)$. (Where we
          use the notation that $\textrm{na}\notin y$ for all $y$.)
    \item $x\textrm{-val}^t: n^t\to  \{x\}\cup \trclos(x)\cup\{\textrm{na}\}$
          such that $x\textrm{-val}^t({c^x}^t)=x$.
          If $x\textrm{-val}^t(a)\notin \ON\cup \{\textrm{na}\}$, then
          ${R^\in}^t(b,a)$ iff
          $x\textrm{-val}^t(b)\in x\textrm{-val}^t(a)$.
          If $x\textrm{-val}^t(a)\in\ON$, then $x\textrm{-val}^t(a)=
          \textrm{ord-val}^t(a)$.
    \item $\mpar\textrm{-val}^t: n^t\to  \{\mpar\}\cup \fclosure(\mpar)
          \cup\{\textrm{na}\}$ such that
          $\mpar\textrm{-val}^t({c^\mpar}^t)=\mpar$ and
          $\mpar\textrm{-val}^t(a)\neq \textrm{na}$ iff ${R^{\fclosure}}^t(a)$.
          If ${R^{\fclosure}}^t(a)$ and ${R^{\dom}}^t(a,b)$, then
          $\textrm{ord-val}^t(b)=\dom(\mpar\textrm{-val}^t(a))$.
    \item $\textrm{rk}^t: n^t\to \delta$
      is a rank-function. I.e., if ${R^\in}^t(a,b)$, then
      $\textrm{rk}^t(a)<\textrm{rk}^t(b)$.
  \end{itemize}

  We set $t\geq_T t'$ if 
  \begin{itemize}
  \item $n^{t'}\geq n^t$, and all the interpretations and functions
      in $t'$ are extensions of the ones in $t$.
      (So we will omit the indices $t$ and $t'$.)
  \item If $i\leq n^t$, and $\varphi_i\in S'$ is the sentence
      $(\forall x_1)\dots(\forall x_l)(\exists y)
      \psi(\vec{x},y)$,
      then for all $\vec a$ in $n^t$ there is a $b\in n^{t'}$ such that
      $t'\vDash \psi(\vec a,b)$. 
  \item Assume that $i<n^t$, $a<n^t$ and $x\textrm{-val}(a)=y\notin
      \ON\cap\{ \textrm{na}\}$.
      Then there is a $b<n_{t'}$ such that $x\textrm{-val}(b)=f^y(i)$,
      cf.~\eqref{eq:myenum1}.
  \end{itemize}

  Then we get the following:
  \begin{itemize}
    \item $T$ is a partial order.
    \item The definition of $T$ can be spelled out
      in $L'$, the definition is absolute,
      and every node of $T$ is element of $L'$. So $T$ is element of $L'$.
    \item In particular $T$ has an infinite descending chain
      in $L'$ iff $T$ has one in $V$.
    \item $T$ has an infinite descending chain iff there is a good
      candidate.
  \end{itemize}
  Let us show just the last item:
  Clearly, a suitable candidate defines an infinite descending 
  chain:
  Given $M$, we can extend it to an $S'$-model (since $S'$ is 
  a conservative extension of $S$) and find a rank function $\rk$ for $M$.
  Then we can construct a chain as a subset of those 
  nodes $t\in T$ that correspond to finite subsets of $M$:
  To every such $t$ we just have to put 
  enough elements
  into $t'$ to witness the requirements.

  On the other hand, a chain defines a candidate:
  The union of the structures in the chain is
  a $\cL'$-structure $M'$ and 
  an $S'$-model. The function $\rk$ defines
  a rank on $M'$. So we can define by induction on this rank
  a function $i:M'\to V$ the following way:
  \[
    i(x)=
    \begin{cases}
      \textrm{ord-val}(a)&\text{if }\textrm{ord-val}(x)\neq \textrm{na}\\
      \{i(y):\, y\in x\}&\text{otherwise.}
    \end{cases}
  \]
  We set $M'=i[M]$. 
  By induction, $i$ is an isomorphism between
  $(M',R^\in,R^\ON,x^{M'},\mpar^{M'})$ and
  $(M,\in,\ON,x,\mpar_M)$, i.e., that
  $M$ is the required good candidate.
\end{proof}

\begin{Rem}
\begin{itemize}
\item
If $\mpar$ is a real, then the transitive collapse of a candidate still is a
candidate. So if $x$ is a real and $S$ as above, the existence of an
appropriate candidate is equivalent to the existence of a transitive candidate,
which is a $\bS12$ statement (in the parameters $\mpar,x,S$).
\item

  There is also
  a notion of non-wellfounded non elementary (nw-nep) forcing, cf.~\cite{MR2151390},
  where candidates do not have to be wellfounded. Then the existence
  of a candidate (with a real parameter) is even a $\bS11$-statement.
\end{itemize}
\end{Rem}


\subsection{Non elementary proper forcing}

We investigate forcing notions $Q$ defined with a nep-parameter $\mpar$:
$Q=\{x:\, \inQ(x,\mpar)\}$.  If $M$ is a $(\ZFCx,\mpar)$-candidate, we assume
that in $M$ the class $\{x:\, \inQ(x,p_M)\}$ is a set, which we will denote by
$Q^M$.  Generally such a $Q^M$ does not have to be a subset of $M$, but to
simplify notation (as in Assumption~\ref{asm:notation1}) we assume that $Q$ is
disjoint to $\ON$ (we can assume  that this requirement is explicitly stated in
the formula $\inQ$).
Then $Q^M\subset M$.
Analogously, we assume that $q\leq p$ iff $\leqQ(q,p,\mpar)$,
and that in $M$, $\{(p,q):\, \leqQ(q,p,\mpar_M)\}$ is a partial order on $Q^M$.
We write $q\leq^M p$ for $M\vDash \leqQ(q,p,\mpar_M)$. Additionally we require
that these formulas are upwards absolute. To summarize:

\begin{Def}\label{def:Qabsdefined}
  \begin{itemize}
    \item
      $M_1$ is a candidate in $M_2$ means the following:
      $M_1$ is a candidate, $M_2$ is either a candidate or
      $M_2=V$, $M_1\in M_2$, and $M_2$ knows that $M_1$ is countable.
    \item
      $\varphi(x)$ is upwards absolute for candidates means:
      If $M_1$ is a candidate in $M_2$, $a\in M_1$, and
      $M_1\vDash \varphi(a)$, then $M_2\vDash \varphi(a)$.
    \item
      A forcing $Q$ is upwards absolutely defined by the 
      nep-parameter $\mpar$, if the following is satisfied:
      \\
      In $V$ and all $(\ZFCx,\mpar)$-candidates $M$,
      $\inQ$ defines a set and
      $\leqQ$ defines a semi partial order on this set, and
      $\inQ$ and $\leqQ$ are upwards absolute for candidates.
  \end{itemize}
\end{Def}

As usual, we define:
\begin{Def}
  $q\in Q$ is $Q$-generic over $M$ (or just: $M$-generic), if
  $q$ forces that (the $V$-generic filter) $G_Q$ is $Q^M$-generic
  over $M$.
\end{Def}
Recall that ``$G$ is $M$-generic'' is defined in \ref{def:MQgen}.
Of course, $G_Q$ will generally not be a subset of $Q^M$.

Note that ``$p\in Q$'', ``$q\leq p$'' and therefore ``$p\comp q$'' are upward
absolute, but $\incomp$ is not. (It will be absolute in most simple
examples of nep-forcing, but typically not in nep-iterations or similar
constructions using nep forcings as building blocks).
This effect is specific for nep forcing, it appears neither in proper forcing
(since for $N\esm H(\chi)$, incompatibility always is absolute), nor in Suslin
proper (since the absoluteness of incompatibility is part of the definition).

Since $\incomp$ is not absolute, ``$q$ is $M$-generic'' is generally 
{\em not} equivalent to ``$q$ forces that all dense $D$ in $M$ meet $G$''. (The
$V$-generic $G$ is not necessarily a $Q^M$-filter.) 

Now we can finally define:
\begin{Def}\label{def:nep} 
  $Q$ is a non elementary proper (nep) forcing for $(\ZFCx,\mpar)$,
  defined by formulas $\inQ(x,\mpar)$, $\leqQ(x,y,\mpar)$, if
  \begin{itemize}
    \item $Q$ is upwards absolutely defined for $(\ZFCx,\mpar)$-candidates, and
    \item for all $(\ZFCx,\mpar)$-candidates $M$ and for all $p\in Q^M$ there is
      an $M$-generic $q\leq p$.
  \end{itemize}
\end{Def}

Sometimes we will denote the $\mpar$ and $\ZFCx$ belonging to $Q$ by $\mpar_Q$
and $\ZFCx_Q$ and denote a  $(\ZFCx_Q,\mpar_Q)$-candidate by
``$Q$-candidate''.

We will only be interested in normal forcings:

\begin{Def}\label{def:qnormal}
  A nep-definition $Q$ is normal, if
  \begin{itemize}
    \item $\ZFCx$ is normal (cf.\ \ref{def:zfcnormal}),
    \item $Q\subseteq \hco$ in $V$ and in all candidates (cf.~\ref{def:hco}),
    \item ``$p\in Q$'' and ``$q\leq p$'' are absolute between $V$ and $H(\chi)$
      (for sufficiently large regular $\chi$).
  \end{itemize}
\end{Def}

If $\ZFCx$ is normal, then the ord-collapse collapse of any $N\esm H(\chi)$
containing $\mpar$ is a candidate. So we get:

\begin{Lem}
  If $Q$ is normal, then for any $p\in Q$ there is a candidate $M$ such 
  that $q\in Q^M$.
  If $Q$ is normal and nep, then $Q$ is proper.
\end{Lem}

\begin{proof}
  This follows directly from Lemma~\ref{lem:basiccandidates} (and the fact
  that in the definition of proper one can assume that
  the elementary submodels contain an arbitrary fixed parameter,
  see e.g. \cite[Def.\ 3.7]{MR1234283}).
\end{proof}

As already mentioned, we are only interested in normal forcings, and we will
later tacitly assume normality whenever we say a forcing is nep.

\begin{Rem}
  However, it might sometimes make sense to investigate non-normal nep
  forcings. Of course such forcings do not have to be proper.
  An example can be found in~\cite[1.19]{MR2115943}:
  We assume CH in $V$, and define a forcing $Q$
  for which we get generic conditions not for all $\ZFC^-$ models,
  but for all models of $2^{\al0}=\al2$. This forcing can collapse
  $\al1$.
\end{Rem}

\subsection{Some simple properties}

Shoenfield absoluteness \ref{lem:shoenfieldlevy}
immediately gives us many simple cases of 
absoluteness. We just give an example:
If $Q$ is upward absolutely defined and normal, then $q\leq p$ is equivalent to
``there is a candidate $M$ thinking that $q\leq p$''. 
So in particular:

\begin{Cor} 
  Assume that $V'$ is an extension of $V$ with the same ordinals, and that
  $Q$ is (normal) nep in $V$ as well as in $V'$.
  Then $p\in Q$, $q\leq p$ and $p\comp q$ are
  absolute between $V$ and $V'$. (But ``$A$ is a maximal antichain'' is 
  only downwards absolute from $V'$ to $V$.) 
\end{Cor} 

The basic theorem of forcing can be formulated as: For a transitive
countable model $M$ and $P$ in $M$
\begin{multline}\label{eq:gut24}
  \left[M\vDash p\forc \varphi(\ntau)\right]\text{ iff } \\
  \left[M[G]\vDash\varphi(\ntau[G])\text{ for every
  $M$-generic filter $G\in V$ containing $p$}\right].
\end{multline}
(And there always is at least one $M$-generic filter $G\in V$ containing $p$.) 

By~\ref{cor:pforc} we get the following: 
\begin{equation}
  \text{If $M$ is a countable, ord-transitive
  model and $P\in M$, then~\eqref{eq:gut24} holds.}
\end{equation}

With the usual abuse of notation, the essential property of proper forcing
can be formulated as follows: If $M$ is an elementary
submodel of $H(\chi)$ and $Q$ in $M$ is proper, then
\begin{multline}\label{eq:gut24r5}
  \left[M\vDash p\forc \varphi(\ntau)\right]\text{ iff }
  \\
  \left[M[G]\vDash\varphi(\ntau[G])\text{ for every
  $M$- and $V$-generic filter $G$ containing $p$}\right].
\end{multline}
(And there always is at least one $M$- and $V$-generic filter  
containing $p$.)

For nep forcings we get exactly the same:
\begin{equation}
  \text{If $Q$ is nep
    and $M$ a $Q$-candidate, then~\eqref{eq:gut24r5} holds.}
\end{equation}

If $M_1$ is a candidate in $M_2$, and $q$ is $Q$-generic over $M_1$, then $q$
does not have to be generic over $M_2$ (since $M_2$ can see more dense sets).
Of course, the other direction also fails:
If $q$ is $M_2$-generic, then generally it is not $M_1$-generic 
(corresponding to the fact in non-ccc proper forcing that not every $V$-generic
filter has to be $N$-generic): $M_1$ could think that $D$ is predense, but
$M_2$
could know that $D$ is not, or $M_1$ could think that $p_1\incomp p_2$, but
$M_2$ sees that $p_1\comp p_2$.  Even for very simple $Q$  satisfying that
$\incomp$ is absolute ``$\{p_i:\, i\in\omega\}$ is a maximal antichain'' need
not be upwards absolute (in contrast to Suslin proper forcing, see example
\ref{ex:ctblnotabs}).

\section{Examples}\label{sec:examples}

There are oodles of examples nep forcings. Actually:
\begin{ruleofthumb}\label{ruleofthumb}
  Every nicely definable forcing notion that can be proven to be proper
  is actually nep.
\end{ruleofthumb}
This rule does not seem to be quite true. A very partial potential 
counterexample is~\ref{ex:cccnotnep}.
However, the rule seems to hold in most cases, and becomes even truer
if the proof of properness uses some form of pure decision and fusion, e.g.,
for $\sigma$-closed or
Axiom A.
(And in these cases, the proof of the nep property is just a trivial
modification of the proof of properness.)

Overview of this section:
\begin{itemize}
  \item Transitive nep forcing: The forcings is a set of reals,
    the definition uses only a real parameter. In this case
    it is enough to consider transitive candidates.
    \begin{enumerate}
      \item[\ref{ss:suslin}] Suslin proper and Suslin$^+$.
      \item[\ref{ss:simplecreatures}] A specific example from the 
        theory of creature forcing.
    \end{enumerate}
  \item Non-transitive nep: The forcings are not subsets of $H(\al1)$,
    and we have to use non-transitive candidates.
    \begin{enumerate}
      \item[\ref{ss:sigma}] Trival examples: $\sigma$-closed forcings.
      \item[\ref{ss:complexcreatures}] 
          Products of creature forcings and similar constructions.
    \end{enumerate}
    Other examples of of non-transitive nep forcings are iterations of nep
    forcings.  We will investigate countable support
    iterations in the Section~\ref{sec:csi}.
  \item Additional topics:
    \begin{enumerate}
      \item[\ref{ss:zap}] 
        Nep,
        creature forcing, and Zapletal's idealized forcing.
      \item[\ref{ss:nonnep}] Counterexamples: forcings that are 
        not nep.
    \end{enumerate}
\end{itemize}

\subsection{Suslin proper forcing}\label{ss:suslin}

Assume that $Q\subseteq \omega\ho$ is defined using a real parameter $\mpar$.

In this case it is enough to consider transitive candidates: Such a candidate
is just a countable transitive model of  $\ZFCx$ containing $\mpar$.\footnote{More
specifically, the straightforward proof shows that in this case 
``$Q$ is nep'' --- i.e.\ ``nep with respect to all
ord-transitive models'' --- is equivalent to: ``$Q$ is nep
with respect to all transitive models''.}

The first notion of this kind was the following:
\begin{Def}\label{def:suslin}
  A (definition of a) forcing $Q$ is Suslin in the real parameter $\mpar$, if
  $p\in Q$, $q\leq p$ and $p\incomp q$ are $\lS11(\mpar)$.
\end{Def}
For Suslin forcings, the nep property is called ``Suslin proper'':
\begin{Def}\label{def:suslinproper}
  \begin{itemize}
    \item $Q$ is Suslin proper, if $Q$ is Suslin and nep. I.e.,
      for every (transitive) candidate $M$ and every $p\in Q^M$ there is an
      $M$-generic $q\leq p$.
    \item $Q$ is Suslin ccc, if $Q$ is Suslin and ccc.
  \end{itemize}
\end{Def}

Suslin ccc implies Suslin proper (in a very strong and absolute way,
cf.~\cite{MR973109}). It
seems unlikely that Suslin plus proper implies Suslin proper, but we do not
have a counterexample.  Cohen, random,
Hechler and Amoeba forcing are Suslin ccc.  Mathias forcing is Suslin proper.

Some forcings are not Suslin proper just because incompatibility is not Borel,
for example Sacks forcing.  This motivated a generalization of Suslin proper,
Suslin$^+$ \cite[p. 357]{MR1234283}.
It is easy to see that every Suslin$^+$ forcing is nep as well, and that many
popular tree-like forcings are Suslin$^+$, e.g., Laver, Sacks and Miller
\cite{MR2252247}.

\subsection{An example of a creature forcing}\label{ss:simplecreatures}

A more general framework for definable forcings is creature forcing, presented
in the monograph \cite{MR1613600} by Ros{\l}anowski and Shelah.  They introduce
many ways to build basic forcings out of creatures, and use such basic forcings
in constructions such as products or iterations.

Typically, the creatures are finite and the basic creature forcing consist of
$\omega$-sequences (or similar hereditarily countable objects made) of
creatures. The proofs that such forcings are proper actually give nep. We
demonstrate this
effect on a specific example (that will also be used in 
Subsection~\ref{ss:zap}).
This specific example is in fact Suslin proper, but other simple (and
similarly defined) creature forcing notions are nep but not Suslin proper.

We fix a sufficiently fast growing\footnote{It is enough to assume
$\mathbf F(i)>2^{i^{(k^*_i)}}$.} 
function $\mathbf F: \omega\to\omega$
and set 
\begin{equation}\label{eq:kstern}
  k^*_i:=\prod_{j<i}\mathbf F(j).
\end{equation}

\begin{Def}\label{def:norm}
  An $i$-creature is a function $\phi:\pow(a)\to \omega$ such that
  \begin{itemize}
    \item
      $a\subseteq \mathbf F(i)$ is nonempty.
    \item  $\phi$ is monotonic, i.e.,
      $b\subset c\subseteq a$ implies $\phi(b)\leq \phi(c)$.
    \item  $\phi$ has bigness, i.e., $\phi(b\cup c)\leq \max(\phi(b),\phi(c))+1$
      for all $b, c\subseteq a$.
    \item $\phi(\emptyset)=0$ and $\phi(\{x\})\leq 1$ for all $x\in a$.
  \end{itemize}
   We set $\val(\phi)\DEFEQ a$, $\nor(\phi)\DEFEQ \phi(a)$, and we
   call $\phi_1$ stronger than $\phi_0$, or: $\phi_1\leq \phi_0$, if
   $\val(\phi_1)\subseteq \val(\phi_0)$ and
   $\phi_1(b)\leq \phi_0(b)$ for all $b\subseteq \val(\phi_1)$.
\end{Def}

For every $\phi$ and $x\in\val(\phi)$ there is a stronger creature $\phi'$ with
domain $\{x\}$. For each $i$, there are only finitely many $i$-creatures.

Another way to write bigness is: 
\begin{equation}\label{eq:big2}
  \text{If }b= c_1 \cup c_2\subseteq a\text{ then either }
  \phi(c_1)\geq \phi(b)-1\text{ or }\phi(c_2)\geq \phi(b)-1.
\end{equation}

\begin{Def}\label{def:creex}
  A condition $p$ of $P$ is a sequence
  $(p(i))_{i\in\omega}$ such that 
  $p(i)$ is an $i$-creature and
  $\liminf_{i\to\infty} \sqrt[\uproot{4}\leftroot{-3}k^*_i]{\nor(p(i))}=\infty$.
  A condition $q$ is stronger than $p$, if 
  $q(i)$ is stronger than $p(i)$ for all $i$.
\end{Def}

Given a $p\in P$, we can define the trunk of $p$ as follows: Let $l$ be maximal
such that  $\val(p(i))$ is a singleton $\{x_i\}$ for all $i<l$. Then the trunk
is the sequence $(x_i)_{i<l}$.

We define the $P$-name $\n\eta$ to be the union of all trunks of conditions in
the generic filter.  For every $n$, the set of conditions with trunk of length
at least $n$ is (open) dense. If $q\leq p$ then the trunk of $q$ extends the
trunk of $p$. So $\n\eta$ is the name of a real, more specifically
$\n\eta\in\prod_{i\in\omega} \mathbf F(i)$.

$P$ is nonempty: For example, the following is a valid condition: 
$\val(p(n))=\mathbf F(n)$, and $p(n)(b)=\lfloor \log_2(|b|) \rfloor$.

\begin{Lem}
  $P$ satisfies fusion and pure decision, so
  $P$ is $\omega^\omega$-bounding and nep (and in particular proper).
\end{Lem}

\begin{proof}[Sketch of proof]
  This is an simple case of \cite[2.2]{MR2214118}. We give 
  an overview of the proof, which
  uses the creature-forcing concepts of bigness and halving:

  {\em Bigness:}
  Assume that $\phi$ is an $i$-creature with $\nor(\phi)>1$,
  and that $F:\val(\phi)\to 2$. Then there is a 
  $\psi\leq \phi$ such that $\nor(\psi)\geq \nor(\phi)-1$
  and such that $F\restriction \val(\psi)$ is constant.

  (This follows immediately from~\eqref{eq:big2}.)

  {\em Halving:}
  Let $\phi$ be an $i$-creature. Then there is an $i$-creature
  $\half(\phi)\leq \phi$ such that 
  \begin{itemize}
     \item $\nor(\half(\phi))\geq \lceil \nor(\phi)/2\rceil$.
     \item If $\psi\leq \half(\phi)$ and $\nor(\psi)>0$, then there is a
       $\psi'\leq \phi$ such that $\nor(\psi')\geq \lceil \nor(\phi)/2\rceil$ and
       $\val(\psi')\subseteq \val(\psi)$.
  \end{itemize}
  (Proof: Define $\half(\phi)$ by
  $\val(\half(\phi))=\val(\phi)$ and
  $\half(\phi)(b)=\max(0,\phi(b)-\lfloor \nor(\phi)/2 \rfloor)$.
  Given $\psi$ as above, we set $b:=\val(\psi)$ and define $\psi'$ by 
  $\val(\psi')=b$ and $\psi'(c)=\phi(c)$ for
  all $c\subseteq b$.  Then
  \[
    0<\nor(\psi)=\psi(b)\leq \half(\phi)(b)=\phi(b)-\lfloor \nor(\phi)/2
    \rfloor,
  \]
  so $\nor(\psi')=\phi(b)\geq \lceil \nor(\phi)/2 \rceil$.)

  {\em Fusion:} We define $q\leq_{m} p$ by: 
  $q\leq p$, $q\restriction m=p\restriction m$, and 
  for all $n\geq m$ either $q(n)$ is equal to $p(n)$
  or $\sqrt[\uproot{4}\leftroot{-3}k^*_n]{\nor(q(n))}>m$.
  If $(p_n)_{n\in \omega}$ is a sequence of conditions
  such that $p_{n+1}\leq_{n+1}p_n$, then there is
  a canonical limit $p_\infty<p_n$.

  Set $\pos(p,n)=\prod_{i<n}\val(p(n))$.
  For $s\in \pos(p,n)$, we construct $p\wedge s\leq p$
  by enlarging the stem of $p$ to be $s$
  (or, if the stem was larger than $n$ to begin with, 
  then the stem extends $s$ and we set $p\wedge s=p$).
  The set $\{p\wedge s:\, s\in\pos(p,n)\}$ is predense
  under $p$. Let $D$ be an open dense 
  set. We say that $p$ essentially is in $D$, if there is an
  $n\in\omega$ such that $p\wedge s\in D$ for all $s\in\pos(p,n)$.

  {\em Pure decision:}
  For $p\in P$, $n\in\omega$ and $D\subseteq P$ open dense there
  is a $q\leq_n p$ essentially in $D$.

  Then the rest follows by a standard argument:

  {\em Nep}: Note that $p\in P$ and $q\leq p$ and  $q\leq_k p$ are Borel
  (so $p\incomp q$ is absolute; actually $\incomp$ is Borel as 
  well, i.e., $P$ is Suslin proper). Fix a transitive model
  $M$ and a $p_0\in P^M$. 
  Enumerate all the dense sets in $M$
  as $D_1,D_2,\dots$. 
  Given $p_n\in M$, pick in $M$ some
  $p_{n+1}\leq_{n+1}p_n$ essentially in $D_{n+1}$.
  In $V$, build the limit $p_\omega\leq p_0$. Then $p_\omega$
  is $M$-generic: Let $G$ be a $P$-generic filter
  over $V$ containing $p_\omega$.
  Fix $m\in\omega$.  We have to show that $G\cap M$ meets $D_m$.
  Note that $G$ contains $p_m$ (since $p_\omega\leq p_m$).
  In $M$, there is an $n\in\omega$ such that 
  $p_m\wedge s\in D_m$ for all $s\in \pos(p_m,n)$.
  The definitions of $\pos(p_m,n)$ as well as $p_m\wedge s$
  are absolute between $M$ and $V$, the set $\pos(p_m,n)$ is
  finite (and therefore subset of $M$), the set
  $\{p_m\wedge s:\, s\in \pos(p_m,n)\}$ is predense (in $V$),
  so $p_m\wedge s\in G$ for some $s\in \pos(p_m,n)$.
  So we get $G\cap M\cap D_m\neq \emptyset$.

  {\em Continuous reading} of
  names and therefore {\em $\omega^\omega$-bounding}
  follows equally easily.

  It remains to show pure decision. Fix $p$, $n$ and $D$ and set $p_0=p$.
  Given $p_m$, we construct $p_{m+1}$ as follows:
  \begin{itemize}
    \item Choose 
\begin{equation}\label{eq:haem}
  h_m>n+m
       \text{ such that }
       \sqrt[\uproot{4}\leftroot{-3}k^*_l]{\nor(p_m(l))}>n+2m
       \text{ for all }
       l\geq h_m. 
\end{equation}
    \item Enumerate $\pos(p_m)$ as $s_1,\dots,s_M$.
      Note that $M\leq k^*_{h_m}$, according to~\eqref{eq:kstern}.
    \item Set $p_m^0=p_m$. Given $p_m^{k-1}$, pick
      $p_m^{k}$ such that
      \begin{itemize}
        \item $\nor(p_m^{k}(l))>(n+2m)^{k^*_l}/2^{k}$ 
          for all $l\geq h_m$.
        \item $p_m^{k}(l)=p_m(l)$ for $l<h_m$.
        \item Either $p_m^{k}\wedge s_{k}$ is essentially in $D$
          (deciding case),
          or it is not possible to find such a condition
          then $p_m^{k}(l)=\half(p_m^{k-1}(l))$ for all $l\geq h_m$
          (halving case).
      \end{itemize}
   \item Set $p_{m+1}$ to be $p_m^M$.
   In particular, $\sqrt[\uproot{4}\leftroot{-3}k^*_l]{\nor(p_{m+1}(l))}>(n+2m)/2$ for all $l\geq h_m$.
 \end{itemize}
 Let $p_\omega$ be the limit of all the $p_m$.
 For every $n\in\omega$ define 
 by downward induction on $h=n,n-1,\dots,h_0$ the $h$-creatures $\phi_{n,h}$
 and sets $\Lambda_{n,h}\subseteq \pos(p_\omega,h)$ in the following way:
 \begin{itemize}
   \item $\Lambda_{n,n}$ is the set of $s\in \pos(p_\omega,n)$
     such that $p_\omega\wedge s$ is essentially in $D$. 
   \item Assume $h_0\leq h<n$.
     So for all $s\in \pos(p_\omega,h)$
     some of the extensions in of $s$ to
     $\pos(p_\omega,h+1)$ will be in $\Lambda_{n,h+1}$
     while others will be not. By shrinking $p_\omega(h)$
     at most $k^*_h$ many times, each time using bigness, we can guarantee that
     the resulting $h$-creature $\phi_{n,h}$ satisfies:
     For all $s\in \pos(p_\omega,h)$ either all extension
     compatible with $\phi_{n,h}$ are in $\Lambda_{n,h+1}$ or
     no extension is. Set $\Lambda_{n,h}$ to be the set of
     those $s$ such that the extensions all are in $\Lambda_{n,h+1}$.
     Note that
     $\sqrt[\uproot{4}\leftroot{-3}k^*_h]{\phi_{n,h}}>1/2
     \sqrt[\uproot{4}\leftroot{-3}k^*_h]{p_\omega(h)}$.
  \end{itemize}
  For each $h$, there are only finitely many possibilities for
  $\Lambda_{n,h}$ and $\phi_{n,h}$, so
  using K\"onig's lemma, we can get a sequence 
  $(\phi_{*,h},\Lambda_{*,h})_{h_0\leq h<\omega}$ such that 
  for all $N$ there is an $n>N$ such that 
  \begin{equation}\label{eq:klwejt}
    (\phi_{*,h},\Lambda_{*,h})=(\phi_{n,h},\Lambda_{n,h})\text{ for all }
    h_0\leq h\leq N.
  \end{equation}

  We claim 
  \begin{equation}\label{eq:whjett}
    \Lambda_{*,h_0}= \pos(p_\omega,h_0)
  \end{equation}
  Then we choose any $n$ such that
  $\Lambda_{n,h_0}=\Lambda_{*,h_0}$ and 
  define $q$ by 
  \[
    q(l)=\begin{cases}
      p_\omega(l) & \text{if } l<h_0\text{ or }l\geq n\\
      \phi_{l,h} & \text{otherwise.}
    \end{cases}
  \]
  Then $q$ essentially is in $D$, according to~\eqref{eq:whjett}
  and the definition of $\Lambda_{n,h}$.

  So it remains to show~\eqref{eq:whjett}. Assume towards
  a contradiction that $s\in \pos(p_\omega )\setminus \Lambda_{*,h_0}$.
  Let $q'$ be the condition with stem $s$ and the creatures
  $(\phi_{*,h})_{h_0\leq h< \omega}$.
  Pick some $r\leq q'$ in $D$.

  Let $s'$ be the trunk of $r$. So $s'$ extends $s$. 
  Let $h$ be the length of $s'$.
  Without loss of generality, we can assume that
  \begin{equation}\label{eq:wiehfr}
    \sqrt[\uproot{4}\leftroot{-3}k^*_l]{\nor(r(l))}>2\text{ for all }l\geq h 
  \end{equation}
  and that $h=h_m$ for some $m$, where $h_m$ is the number picked
  in~\eqref{eq:haem} to construct $p_{m+1}$.
  In particular, $s'=s_k$ for some $k$, so
  \begin{equation}\label{eq:wiu7}
    r\leq p^k_m
  \end{equation}
  We know that $r\in D$. This implies that
  \begin{equation}\label{eq:wiu8}
    r'\DEFEQ p^k_m\wedge s_k\text{ essentially is in }D
    \text{ (and }r\leq r'\text{).}
  \end{equation}
  Assume otherwise. Then pick $H>h_m$ such that
  $\sqrt[\uproot{4}\leftroot{-3}k^*_l]{\nor(r(l))}>(n+2m)/2$
  for all $l\geq H$. For $h_m\leq l<H$, we can unhalve
  $r(l)$ to get some $\tilde r(l)$ with norm at least
  $\nor(p^{k-1}_m)/2> (n+2m)^{k^*_l}/2^k$.
  Then the condition consisting of trunk $s'$,
  the creatures $\tilde r(l)$ for  $h_m\leq l<H$
  and $r(l)$ for $l\geq H$ would be a suitable condition for
  the deciding case, a contradiction to the fact that we are
  in the halving case.
  This shows~\eqref{eq:wiu8}.

  Note that $p_\omega \wedge s'\leq  p^k_m\wedge s'$, so 
  by~\eqref{eq:wiu8} we get that $p_\omega \wedge s'$ essentially is in
  $D$. 
  We can  now derive the desired contradiction:
  \begin{equation}\label{eq:wtqwtwrqwr}
    p_\omega \wedge s'\text{ is not essentially in }D.
  \end{equation}
  Proof: Assume otherwise, i.e., for some $N$
  every $s''\in\pos(p_\omega,N)$ extending $s'$ 
  is in $D$. Pick $n>N$ as in~\eqref{eq:klwejt}.
  Then according to the definition of $\Lambda_{n,h}$,
  we get $s'\in \Lambda_{n,h_m}$ and therefore 
  $s\in \Lambda_{n,h_0}$, a contradiction.
  This shows~\eqref{eq:wtqwtwrqwr}.
\end{proof}

\subsection{$\sigma$-closed forcing notions}\label{ss:sigma}
The simplest (and not very interesting) examples of non-transitive
nep-forcings are the $\sigma$-closed ones. We use the following 
obvious fact:

\begin{Fact}
  Assume that $Q$ is upwards absolutely defined, that
  $\incomp$ is upwards absolute as well (and therefore absolute)
  and that $Q$ is $\sigma$-closed (in $V$). Then $Q$ is nep.
\end{Fact}


It is not enough to assume that $Q$ is ccc (in $V$ and all candidates) instead
of $\sigma$-closed, see Example~\ref{ex:cccnotnep}.

So the following definition 
of $Q=\{f:\omega_1\to \omega_1\text{ partial, countable}\}$ is nep:
\begin{Exm} Define $Q$ by $\mpar=\omega_1$ and $f\in Q$ if
  $f: \mpar \to \mpar$ is a countable partial function.
  Then $Q$ is nep.
\end{Exm}

Note that we cannot use $\omega_1$ in the definition directly, since there are
candidates $M$ such that $\om1^M>\om1^V$.  Neither could we use $f:\alpha\to
\mpar$, $\alpha\in\mpar$, since such an $f$ in a candidate $M$ really has
domain $\alpha\cap M$, which is generally not an ordinal (i.e.,
this definition would not be upwards absolute).

More generally, we can get the examples:
\begin{Exm}
  Assume that $\mpar$ codes the ordinals $\kappa^\mpar$ and $\lambda^\mpar$, and
  set $Q=\{f:\kappa^\mpar\to\lambda^\mpar\text{ partial, countable}\}$
  (ordered by extension). Then $Q$ is nep.
\end{Exm}

This example shows that a nep forcing can look completely different in
different candidates: Assume $\kappa^\mpar=\om1$ and $\lambda^\mpar=\om2$.
So in $V$, $Q$ collapses $\om2$ to $\om1$.
Let $N\esm H(\chi)$, $M=\ordcol(N)$, and $M_0\in V$ a
forcing-extension of $M$ for the collapse of $\om1$ to $\omega$.
Then
$M_0$ is a candidate, and $M_0$ thinks that $\om1^V$ is countable,
so $Q$ is trivial in $M_0$.  
If $M_1\in V$ is a forcing-extension  of $M$
for the collapse of $\om2$ to $\om1$, then in $M_1$
$Q$ is isomorphic to the set of countable partial functions 
from $\om1$ to $\om1$.

A slight variation (still $\sigma$-closed):
\begin{Exm}\label{ex:ctblnotabs}
  Set $\mpar=\om1$, 
  $Q=\{f:\mpar\to L\cap 2\ho\text{ partial, countable}\}$
  (ordered by extension).
  Then $Q$ is nep, and there is a candidate $M$ which thinks that
  $A$ is a countable maximal antichain of $Q^M$, 
  but $A$ is not maximal in $V$.
\end{Exm}

\begin{proof}
  $x\in L$ is upwards absolute, so $\in_Q$, $\leq_Q$ and $\incomp_Q$ are
  upwards absolute. Clearly $Q$ is $\sigma$-closed in $V$. So $Q$ is nep.
  Assume $V=L$, and pick some
  $N\esm L_\kappa$ for $\kappa$ regular.
  Set $M=\ordcol(N)$.
  In $L$, construct $M'$ as a forcing-extension of $M$ for the collapse of
  $\om1$
  to $\omega$. Then $M'$ thinks $L\cap  2\ho$ is countable, i.e., that
  $\{(0,x):\, x\in L\cap 2\ho\}$
  is a countable maximal antichain.
\end{proof}

Another, trivial example for a countable antichain with non-absolute maximality
is the (trivial) forcing defined by $Q=\{1_Q\}\cup (L\cap 2\ho)$ and $x\leq y$
iff $y=1_Q$ or $x=y$.

\subsection{Non-transitive creature forcing}\label{ss:complexcreatures}

Some creature forcing constructions use a countable support product (or a
similar construction) built from basic creature forcings.
In the useful cases these forcings can be shown to be proper, and the proof
usually also shows nep. One would take the index set of the product to be
an ordinal $\kappa$, and choose the nep parameter
$\mpar$ with domain $\kappa$ such that 
$\mpar(\alpha)$ is the nep-parameter (a real) for the basic creature forcings
$Q_\alpha$.

To give the simplest possible example:
\begin{Lem}\label{ex:sacks}
  The countable support product (of any size) of Sacks forcings is nep.
\end{Lem}
\begin{proof}
  Again, the standard proof of properness works.
  First some notation:
  A splitting node is a node that has two immediate successors.
  The $n$-th splitting front $F^T_n$ of a perfect
  tree $T\subseteq 2^{<\omega}$ is 
  the set of splitting nodes $t\in T$ such that $t$ has 
  exactly $n$ splitting nodes 
  below it. Note that $F^T_n$ is a front (i.e., it meets
  every branch) and therefore finite (since $T$ has finite splitting).
  Let $\kappa$ be the index set of the product. So a condition
  $p$ consists of a countable domain $\dom(p)\subseteq \kappa$ and
  for every $i\in\dom(p)$ a perfect tree $p(i)$.
  In particular, $q\leq p$ means $\dom(q)\supseteq \dom(p)$
  and $q(i)\subseteq p(i)$ for all $i\in\dom(p)$.
  \begin{itemize}
    \item
      For $u\subseteq \kappa$ finite, 
      $q\leq_{n,u} p$ means: 
      $q\leq p$, and 
      $F_n^{p(i)}=F_n^{q(i)}$ for all $i\in u$.
    \item 
      Fusion: If we use some simple bookkeeping, we can guarantee
      that a sequence $p_{n+1}\leq_{n,u_n}p_n$ has a limit $p_\omega$.
      (It is enough to make sure that the $u_n$ are
      increasing and that $\bigcup_{n\in\omega}u_n$
      covers $\dom(p_\omega)$.)
    \item
      For $u\subseteq \dom(p)$ finite,
      we set $\pos_u(p,n)=\prod_{i\in u} F_n^{p(i)}$ (a finite set).
      For $\eta\in\pos_u(p,n)$ there is a canonical
      $p\wedge \eta \leq p$  defined in the obvious way (we
      increase some trunks).
    \item
      Pure decision: given a condition $p$,
      some finite $u\subseteq \dom(p)$,
      some $n\in\omega$ and an open dense set $D$, 
      we can
      strengthen $p$ to some $q\leq_{n,u}p$
      such that $q\wedge \eta\in D$ for all $\eta\in\pos_u(q,n)$.

      To show this, just enumerate $\pos_u(q,n+1)$ as
      $\nu_0,\dots,\nu_{M-1}$,
      set $p_0=p$, given $p_m$ find  $p'\leq p\wedge \nu_m$
      in $D$ and then set $p_{m+1}$ to be $p'$ ``above $\nu_m$''
      and $p_m$ ``on the parts incompatible with $\nu_m$''.
      Then set $q=p_M$.
    \item
      This implies nep: 
      Let the forcing parameter $\mpar$ code $\kappa$
      (e.g., $\mpar:\kappa\to\{0\}$).
      Then we can define 
      $P$ to consist of all countable partial functions $p$ with domain 
      $\dom(\mpar)$ such that $p(\alpha)$ is a perfect tree for all 
      $\alpha\in\dom(p)$. This is an absolute definition, and 
      compatibility is absolute.

      Fix $p=p_0\in M$.
      Enumerate as $D_0,D_1,\dots$ all sets in $M$ such that
      $M$ thinks $D_i$ is dense. 
      Given $p_{m-1}\in M$, pick a suitable $u_{m}$ and find in $M$ some
      $p_{m}\leq_{u_{m},m} p_{m-1}$ such that $p_{m}\wedge s\in D_{m}$
      for all $s\in\pos_u(p,m)$. In $V$, fuse the sequence into
      $p_\omega$. Then $p_\omega\leq p$ is $M$-generic: 

      Assume that $G$ contains $p_\omega$ an therefore $p_m$.
      We know that $p_m \wedge s$ is in $G$ for some
      $s\in \pos_{u_m}(p_m,m)$. Then $p_m \wedge s\in D_m\cap M\cap G$.
    \item With similar standard arguments we get $\omega^\omega$-bounding.
      \qedhere
  \end{itemize}
\end{proof}

\subsection{Idealized forcing}\label{ss:zap}

Zapletal \cite{MR2391923} developed the theory of (proper) forcing notions of the
form $P_I=\textrm{Borel}/I$ for (definable) ideals $I$. (A smaller set is a 
stronger condition.) The generic filter $G_I$ of such forcing notions
is always determined by a canonical generic real $\n\eta_I$. 
How does nep and creature forcing fit into this framework?

\begin{itemize}
  \item According to the Rule of Thumb \ref{ruleofthumb},
    most $P_I$ which can be shown to be proper, are in fact nep.
    But we do not know of any particular theorems or counterexamples.
  \item In particular, we do not know whether there is a good
    characterization of
    the (definable) ideals $I$ such that $P_I$ is nep.
    (Even assuming that $P_I$ is proper, which is a tricky 
    property in itself, cf.~\cite[2.2]{MR2391923}.)
  \item Most nicely definable forcing notions 
    with hereditarily countable conditions such that the
    generic object is determined by a real are equivalent to some $P_I$,
    and \cite{MR2391923} proves several theorems in that direction. 
    (E.g., in many ccc cases there is a natural generic real,
    and the ideal $I$ can be taken to consist of those Borel sets
    that are forced not to contain the generic.)
    However, there are natural examples of
    creature forcings where the generic filter is determined by
    a generic real and yet the forcing is not of the form $P_I$.
    The next lemma gives an example.
  \item Many of the nice consequences that we get for (transitive) nep forcings
    also follow for forcings of the form $P_I$
    (not assuming nep, but sometimes
    other additional properties).
    For example the preservation Theorem~\cite[9.4]{MR2155272}
    mentioned in the introduction corresponds to~\cite[6.3.3]{MR2391923}.
\end{itemize}

The following lemma is due to Zapletal.\footnote{Jind{\v{r}}ich Zapletal,
personal communication, November 2007.}

\begin{Lem}\label{lem:zapslem}
  Let  $P$ be the forcing of subsection~\ref{ss:simplecreatures}.
  \begin{enumerate}
    \item The generic filter $G$ is determined by the generic real $\n\eta$.
    \item $(P,\n\eta)$ is not equivalent to a forcing of the form
      $(P_I,\n\eta_I)$.
  \end{enumerate}
\end{Lem}

To make this precise, we have to specify what we mean with ``equivalent''.
We use the following version:
\begin{Def}\label{def:forceq}
 A forcing notion $P$ together with the $P$-name $\n\eta$
 are equivalent to $P_I$ (with the canonical generic real $\n\eta_I$),
 if there are $P$-names $G_I'$ and $\n\eta'_I$ and
  a $P_I$-name $G'$ such that $P$ forces:
  $G_I'$ is the $P_I$-generic filter over
  $V$ corresponding to the generic real $\n\eta'_I$, and $G'[G'_I]_{P_I}=G$.
\end{Def}
I.e., we can reconstruct the $P$-generic filter $G$ by evaluating the
$P_I$-name $G'$ with the $P_I$-generic filter $G'_I$.

In particular, this implies
\begin{equation}\label{eq:jhetw}
  (\forall p\in P)\,(\exists q\leq p)\, (\exists  \tilde B\in P_I)\
    q\forc_P \tilde B\in G'_I\ \&\  \tilde B\forc_{P_I} p\in G'.
\end{equation}

We will need the following straightforward fact:
\begin{Lem}\label{lem:forceq}
  Assume that $(P,\n\eta)$ is equivalent 
  to $P_I$, and that there is a Borel function $f$ such that
  $\forc_P \n\eta'_I=f(\n\eta)$.
  Then the canonical map $\varphi: P_J\to \ro(P)$ 
  defined by $B\mapsto \llbracket \n\eta\in B\rrbracket_P$ 
  is a dense embedding, where we set
  $J=\{B:\, \forc_P \n\eta\notin B^{V[G]}\}$.
\end{Lem}
\begin{proof}
  Given $p\in P$, we need some $B$ such that $0\neq \llbracket
  \n\eta\in B\rrbracket\leq p$. Let $q,\tilde B,q$ be as in~\eqref{eq:jhetw},
  and set $B=f^{-1}\tilde B$. In particular $\tilde B\forc_{P_I} p\in G'$,
  so 
  \[
    \n\eta\in B\text{ iff }
    f(\n\eta)=\in \tilde B\text{ iff }
    \n\eta'_I\in \tilde B\text{ iff }
    \tilde B\in G'_I\text{, which implies }
    p\in G'[G'_I],
  \]
  i.e., $p\in G$. Also, $q\forc_P \tilde B\in G'_I$, so
  $q\leq \llbracket \n\eta\in B\rrbracket\leq p$.
\end{proof}

A density argument together with \cite[3.3.2]{MR2391923}
gives the following:
\begin{Lem}\label{lem:gurkelem8}
  Assume that $P$ is $\omega^\omega$ bounding and has
  Borel reading of names with respect to the $P$-name
  $\n\eta$ and that $(P,\n\eta)$ is equivalent to 
  $P_I$. Fix $p_0\in P$. Then there is a $p_1\leq p_0$
  such that $P'=\{p\in P:\, p\leq p_1\}$ satisfies the 
  following:
  For all $p$ there is a 
  compact set $C$ such that
  $0\neq \llbracket \n\eta\in C \rrbracket_{\ro(P')}\leq p$.
\end{Lem}
Borel reading means: For all $P$-names
$\n r$ for a real and all $p\in P$ there is a Borel function $f$ and a $q\leq
p$ forcing that $\n r=f(\n\eta)$.

Note that the forcing of Subsection~\ref{ss:simplecreatures} has
Borel reading (even continuous reading) of names from the canonical
generic $\n\eta$.

\begin{proof}
  Given $p_0\in P$, there is some $p_1\leq p$ and $f$ Borel
  such that $p_1$ forces $\n\eta'_I$ to be $f(\n\eta)$.
  So according to Lemma~\ref{lem:forceq}, the
  canonical embedding $\varphi: P_J\to \ro(P')$ 
  is dense for $J=\{B:\, p_1\forc_P \n\eta\notin B^{V[G]}\}$
  and $P'=\{p\leq p_1\}$.
  Given $p\in P'$, find some Borel-code $B$ such that $\varphi(B)\leq p$.
  \cite[3.3.2]{MR2391923} gives a $J$-positive compact subset of $B$.
\end{proof}

\begin{proof}[Proof of \ref{lem:zapslem}]
  Proof of (1).

  We will use the following property of norms, cf.\ Definition~\ref{def:norm}:
  \begin{equation}\label{eq:join}
    \parbox{0.8\textwidth}{For norms $\phi_0,\phi_1$ with
      $\val(\phi_0)\cap\val(\phi_1)\neq\emptyset$ there is a weakest norm
      $\phi_0\wedge \phi_1$ stronger than both $\phi_0$ and $ \phi_1$.}
  \end{equation}
  Proof: We define $\psi=\phi_0\wedge \phi_1$ the following way:
  $\val(\psi)=\val(\phi_0)\cap\val(\phi_1)$  and $\psi(b)$ is defined
  by induction on the cardinality of $b$: If $|b|\leq 1$, then
  $\psi(b)=\min( \psi(b)=\min\phi_0(b),\phi_1(b))$. Otherwise,
  $\psi(b)=\min(X(b))$, for
  \[
     X(b)=\{\phi_0(b),\phi_1(b)\}
             \cup\{1+\max(\psi(b_0),\psi(b_1)):\, b_0\cup b_1=b \}.
  \]
  We have to show that $\psi$ is a norm: Bigness follows imediately
  from the definition. It remains to show monotonicity. We show by
  induction on $b$:
  \[
    (\forall c\subseteq b)\, \psi(c)\leq \psi(b)
  \]
  I.e., $(\forall m\in X(b))\, \psi(c)\leq m$.
  For $m=\phi_0(b)$, we have $\psi(c)\leq \phi_0(c)\leq  \phi_0(b)=m$.
  The same holds for $m=\phi_1(b)$. 
  So assume $m=1+\max(\psi(b_0),\psi(b_1))$, without loss of generality
  for nonempty and disjoint $b_0,b_1$.
  Then $b_0\cap c\subsetneq b $ and $b_1\cap c\subsetneq b $, so by
  definition $\psi(c)\leq 1+\max(\psi(b_0\cap c), \psi(b_1\cap c))$
  which is (by induction) at most $1+\max(\psi(b_0), \psi(b_1))=m$.

  On the other hand it is clear that $\psi$ is the biggest possible
  norm that is smaller than $\phi_0$ and $\phi_1$. So we get~\eqref{eq:join}.

  We will also need:
  \begin{equation}\label{eq:wq53}
    (\forall b\subseteq \val(\phi_0 \wedge \phi_1))\,
    (\exists b_0,b_1)\,
    b=b_0\cup b_1\ \&\ 
    (\phi_0 \wedge \phi_1)(b)\geq \max(\phi_0(b_0),\phi_1(b_1))
  \end{equation}

  Proof: Again, write $\psi$ for $\phi_0 \wedge \phi_1$.
  By induction on $|b|$: If $\psi(b)=\phi_0(b)$,
  we can set $b_0=b$ and $b_1=\emptyset$. Analogously for $\psi(b)=\phi_1(b)$.
  If $\psi(b)=1+\max (\psi(c_0),\psi(c_1))$ for
  $c_0\subsetneq b$ and $c_1\subsetneq b$, then by induction 
  $\psi(c_0)\geq \max(\phi_0(d^0_0),\phi_1(d^1_0))$ and
  $\psi(c_1)\geq \max(\phi_0(d^0_1),\phi_1(d^1_1))$, so we 
  can set $b_0=d^0_0\cup d^0_1$ and $b_1=d^1_0\cup d^1_1$.
  Then 
  \begin{align*}
    \psi(b)=&1+\max(\psi(c_0),\psi(c_1))\geq
      1+\max(\phi_0(d^0_0),\phi_0(d^0_1),\phi_1(d^1_0),\phi_1(d^1_0))\geq \\
            &\phi_0(d^0_0\cup d^0_1)
  \end{align*}
  (because of bigness of $\phi_0$), and analogously 
  $\psi(b)\geq \phi_1(d^1_0\cup d^1_1)$.
  This shows~\eqref{eq:wq53}.
  
  For compatible $p,q\in P$ we can define $p\wedge q$ by $(p\wedge
  q)(i)=p(i)\wedge q(i)$.
  This is the weakest condition stronger than both $p$ and $q$. 
  An immediate consequence of~\eqref{eq:join} is:  $p\incomp q$
  is equivalent to
  \begin{multline}\label{eq:wkljtq}
      (\exists n\in\omega)\, \val(p(n))\cap \val(q(n)))=\emptyset
      \text{ or }\\
      (\exists b\subseteq \omega \text{ infinite})\, (\exists M\in\omega)\,
      (\forall n\in b)\, \nor(p(n)\wedge q(n))<M^{k^*_n}.
  \end{multline}

  An obvious candidate for reconstructing the generic filter
  $G$ from the generic real $\n\eta$ 
 (that works with many tree-like forcings) would be the set
  \[
    H_0=\{p\in P:\, \n\eta \in\prod_{n\in\omega}\val(p(n))\}.
  \]
  However, due to the halving property of $P$, this fails miserably:
  There are incompatible conditions $q$ and $r$ with 
  $\val(q(n))=\val(r(n))$ for all $n$. More specifically, we get
  the following: For all $p$
  there is an $r\leq p$ such that
  \begin{equation}\label{eq:incomp}
    r\incomp \half(p),\text{ and }\val(r(n))\subseteq \val(\half(p)(n))\text{ for
    all }n.
  \end{equation}
  Proof: Set $q(n)=\half(p)$. Pick for all sufficiently large $n$ some
  $a_n\subseteq \val(q(n))$ such that $q(n)(a_n)=2$.
  Using the halving property, we can find for all $n$
  some $\phi_n\leq p(n)$ such that $\val(\phi_n)\subseteq a_n$
  and $\nor(\phi_n)>\nor(p(n))/2$. Set $r=(\phi_n)_{n\in\omega}$.
  Then $r$ and $q$ cannot be compatible, since 
  $q(n)(\val(r(n)))$ is bounded. This shows~\eqref{eq:incomp}.

  Back to the proof. First note the following:
  Fix $p\in P$. Let $X(p)$ be the set of all sequences 
  $\bar b=(b_n)_{n\in d}$ where $d$ is an
  infinite subset of $\omega$ and
  $b_n\subseteq \val(p(n))$ such that
  $\{\sqrt[\uproot{4}\leftroot{-3}k^*_n]{p(n)(b_n)}:\, n\in d\}$ is bounded.
  Fix some $\bar b\in X(b)$.
  Then $p$ forces that $\n\eta$ is not in the set 
  \begin{equation}\label{eq:gurke1}
    \n N_{p,\bar b}=\{\nu\in \prod_{n\in\omega}\val(p(n)):\, (\exists^\infty
    n\in d) \nu(n)\in b_n\}.
  \end{equation}
  Proof: Assume towards a contradiction that some  $p'\leq p$ 
  forces $\n\eta\in \n N_{p,\bar b}$.
  So there is a bound $M$ such that
  $p'(n)(b_n)<M^{k^*_n}$ for all $n\in d$. Fix $N(n)$ such that 
  $\nor(p'(l))>(n+1+M)^{k^*_l}>1+(n+M)^{k^*_l}$ for all $l>N(n)$. 
  For all $l>N(n)$ we get
  $p'(l)(\val(p'_l)\setminus b_l)> (n+M)^{k^*_l}$ (by bigness).
  Let $p''$ be the condition $p'(l)\restriction (\val(p'_l)\setminus b_l)$
  for $l\in (b\setminus N(0))$. Then $p''$ forces that $\n\eta\notin \n
  N_{p,\bar b}$, a contradiction. This shows~\eqref{eq:gurke1}.

  We claim that the following defines $G$:
  \begin{equation}\label{eq:hj35}
    H=H_0\cap 
    \{p\in P:\, 
    (\forall \bar b\in X(p)\cap V)\ \n\eta \notin \n N_{p,\bar b}\}.
  \end{equation}
  $H\supseteq G$ by \eqref{eq:gurke1}, so it is enough to show that
  all $p_1,p_2\in H$ are compatible.
  Set $b_n=\val(p_1(n))\cap \val(p_2(n))$. Note that $b_n$ is nonempty,
  since $p_1,p_2\in H_0$. So according to~\eqref{eq:wkljtq} we can
  assume towards a contradiction that the following holds (in $V$):
  \[
    (\exists b\subseteq \omega \text{ infinite})\, (\exists M\in\omega)\,
    (\forall n\in b)\, \nor(p(n)\wedge q(n))<M^{k^*_n}.
  \]
  According to~\eqref{eq:wq53}, we get 
  $c^1_n,c^2_n$ such that $c^1_n\cup c^2_n=b_n$
  and
  $p_i(n)(c^i_n)<M^{k^*_n}$ for $n\in b$ and $i\in\{0,1\}$.
  We assumed that $\n\eta\notin \n N_{p_1,\bar c^1}$, i.e.,
  $\n\eta(n)\in c^1_n$ for only finitely many $n$. The same is
  true for $c^2_n$, a contradiction. This shows~\eqref{eq:hj35}
  and therefore item~(1).

  Note that 
  to construct $G$ from $\n\eta$, we use the (complicated) set
  $(2^\omega)^V$; compare that with the much easier construction
  of $H_0$.

  Proof of (2).

  Let us assume towards a contradiction that 
  $P$ is equivalent to $P_I$. So it satisfies the 
  assumptions of Lemma~\ref{lem:gurkelem8}.
  Fix $p\in P'$, and set 
  $q=\half(p)$. Let $C$ be compact such that 
  \begin{equation}\label{eq:gurke553}
    0\neq\llbracket \n\eta\in C\rrbracket<q.
  \end{equation}
  Then $\prod_{n\in\omega} \val(q(n))\subseteq C$, since
  $C$ is closed. Let $r\leq p$ be incompatible to $q$
  such that $\val(r(n))\subseteq \val(q(n))$ as in~\eqref{eq:incomp}.
  Then $\prod_{n\in\omega} \val(r(n))\subseteq C$, therefore
  $r\forc \n\eta\in C$. So $r\leq^* q$ by~\eqref{eq:gurke553},
  which contradicts $r\incomp q$.
\end{proof}


\subsection{Counterexamples}\label{ss:nonnep}

Being nep is a property of the definition, not the forcing.  Of course 
we can find for any given proper forcing a definition which is not nep
(take any definition that is not upwards absolute).
For the same trivial reasons, a forcing ``absolutely equivalent'' to a nep
forcing doesn't have to be nep itself. For example:
\begin{Exm}\label{ex:densenotnep}
  There are upward absolute definitions of (trivial) forcings $P$, $Q$ s.t.  in
  $V$ and all candidates, $P$ is a dense suborder of $Q$, $P$ is nep but $Q$ is
  not nep.
\end{Exm}

\begin{proof}
  Pick $\mpar\in L\cap 2^\omega$ and a candidate $M_0$ that thinks
  $\mpar\notin L$.  Define $P=\{1,p_1,p_2\}$, $x\leq_P y$ iff $y=1$ or $x=y$.
  Set
  $Q=P\cup\{q_1,q_2\}$ and define the order on $Q$ by:
  $1\leq q_i\leq p_i$, and if $\mpar\in L$,
  then also $p_2\leq q_1$ and $p_1\leq q_2$.  These definitions are upwards
  absolute  and $P$ is nep.  However,
  $M_0\vDash\qb q_1\incomp q_2\qe$. But every $Q$-generic Filter over $V$
  contains $q_1$ and $q_2$, so there cannot be a $Q$-generic condition over
  $M_0$.
\end{proof}

If $Q$, $\leq$ and $\incomp$ are $\lS11$ and $Q$ is ccc, then $Q$ is Suslin
ccc, and therefore (transitive) nep.  (One of the reasons is that in the
$\lS11$-case it is absolute for countable antichains to be maximal.)
This is not true anymore if the definition of $Q$ is just $\lS12$:

\begin{Exm}\label{ex:cccnotnep}
  Let $Q$ be random forcing in $L$ ordered by inclusion, i.e.,
  \[
    Q=\{r\in L:\, r\text{ is a Borel-code for a non-null-set}\}.
  \]
  Then $p\in Q$ is $\lS12$ and
  $q\leq p$ and $p\incomp q$ are (relatively) Borel, and
  in $V$ and all candidates $Q$ is ccc.
  But $Q$ is not nep.
\end{Exm}

\begin{proof}
  Pick in $L$ a (transitive) candidate
  $M$ such that $M$ thinks that 
  $\om1^L$ (and therefore $Q$) is countable.
  In particular there is for each $n\in\omega$ 
  a maximal antichain $A_n$ in $M$
  such that $\mu(X_n)<1/n$ for $X_n=\bigcup_{a\in A_n}a$.
  (Of course $M$ thinks that $X_n$ is not in $L$. But really
  it is, simply because $M\subseteq L$.)
  Take any $q\in Q^V$, and pick $n$ such that $1/n<\mu(q)$.
  Then $q'=q\setminus X_n$ is positive and in $L$, and
  a generic filter containing $q'$ does not meet the antichain $A_n$.
\end{proof}

It is however not clear whether $Q$ could not have another definition that is
nep, or at least whether $Q$ is forcing-equivalent to a nep forcing.  If $L$ is
very small (or very large) in $V$, then $Q$ is Cohen (or random, respectively)
and thus equivalent to a nep forcing notion.  If $V'$ is an extension of $V=L$
by a random real, then in $V'$ the forcing $Q$ (which is ``random forcing in
$L$'') seems to be more complicated (it adds an unbounded real, but no Cohen).
We do not know whether in this case $Q$ is equivalent to a nep forcing.

\section{Countable support iterations}\label{sec:iteration}\label{sec:csi}

This section consists of three subsections:
\begin{enumerate}
  \item[\ref{ss:nonabs}] 
    We introduce the basic notation and preservation theorem.
    We get generic conditions for the limit, but not an
    upwards absolute definition of the forcing notion.
  \item[\ref{ss:absolute}]
    We introduce an equivalent definition of the iteration
    which is upwards absolute. So the limit is again nep.
  \item[\ref{ss:subsetiteration}]
    We modify the notions of Subsection~\ref{ss:nonabs}
    to subsets of the ordinals, and give a nice application.
\end{enumerate}

For this section, we fix a sequence $(Q_\alpha)_{\alpha\in \epsilon}$ of
forcing-definitions and a nep-parameter $\mpar$ coding the parameters
$(\mpar_\alpha)_{\alpha\in\epsilon}$, i.e., $\mpar$ is a nep-parameter with
domain $\epsilon$ and
$\mpar(\alpha)$ is the
nep-parameter used to define $Q_\alpha$
for each $\alpha\in\epsilon$.
(So we assume that the sequence of defining formulas
and parameters live in the ground model.)

To further simplify notation, we also assume that candidates are
successor-absolute, i.e., ``$\alpha$ is successor'' and
the function $\alpha\mapsto \alpha+1$ are absolute for all candidates.
\begin{Rem}
  This assumption 
  is not really necessary. Without it, we just have
  to use ``$M$ thinks that $\alpha=\zeta+1$'' instead of 
  just ``$\alpha=\zeta+1$'' in the definition of $G_\alpha^M$ etc., 
  similarly to~\ref{def:iterationalongsubsets}. 
\end{Rem}

Also, we assume the following (which could be replaced by weaker 
conditions, but is satisfied in practice anyway):
\begin{itemize}
  \item In every forcing extension of $V$, each $Q_\alpha$ is 
    normal nep (for $\ZFCx$ candidates).
  \item We only start constructions 
    with candidates $M$ 
    such that generic extensions $M[G]$ satisfy $\ZFCx$.\footnote{%
  Formally we can require that $M$ satisfies some stronger $\ZFC'$
  and that $\ZFC'$ proves 
  that every formula of $\ZFCx$ is forced by
  all countable support iterations of forcings of the form $Q_\alpha$.
  Also, we assume that ZFC proves that $H(\chi)$ satisfies $\ZFCx$
  for sufficiently large regular $\chi$, and that ZFC proves 
  that the defining formulas are absolute between $V$ and $H(\chi)$.}
\end{itemize}

\subsection{Properness without absoluteness}\label{ss:nonabs}

We use the following notation:
For any forcing notion, $q\leq^* p$ means $q\forc p\in G$.
\begin{Def}
  Let $M$ be a candidate.
  \begin{itemize}
    \item $P_\beta$ is the countable support iteration
      (in other terminology: the limit of)
      $(P_\alpha, Q_\alpha)_{\alpha\in\beta}$ (for all $\beta\leq \epsilon$).
      We use $G_\alpha$ to denote the $P_\alpha$-generic filter over $V$,
      and $G(\alpha)$ for the $Q_\alpha$-generic filter over $V[P_\alpha]$.
    \item $P^M_\beta$ is the element of $M$ so that $M$ thinks: $P^M_\beta$ is
      the countable support iteration of the sequence $(Q_\alpha)_{\alpha\in
      \beta}$ (for $\beta\in \epsilon\cap M$).
  \end{itemize}
\end{Def}

In certain $P_\epsilon$-extensions of $V$ the generic filter $G$
defines a canonical $P^M_\epsilon$-generic $G^M_\epsilon$ over $M$:
\begin{Def}\label{def:PMalpha}
  Given $G\subset P_\epsilon$, we define
  $G^M_{\alpha}$ by induction on $\alpha\in \epsilon\cap M$
  by using the following definition, provided it results in a
  $P^M_\alpha$-generic filter over $M$. In that case 
  we say ``$G$ is $(M,P_\alpha)$-generic''. Otherwise,
  $G^M_{\alpha}$  (and $G^M_{\beta}$ for all $\beta>\alpha$) are undefined.
  \begin{itemize}
    \item If $\alpha=\zeta+1$, then
      $G^M_\alpha$ consists of all 
      $p\in P^M_\alpha$ such that 
      $p\restriction \zeta\in G^M_\zeta$
      and $p(\zeta)[G^M_\zeta]\in G(\zeta)$.
    \item If $\alpha$ is a limit, then 
      $G^M_\alpha$ is the set of all $p\in P^M_\alpha$ such that
      $p\restriction \zeta\in G^M_\zeta$ for all $\zeta\in \alpha\cap M$.
  \end{itemize}
\end{Def}

\begin{Def}
  \begin{itemize}
    \item Assume that $G$ is  $(M,P_\alpha)$-generic and
      $\zeta\in\alpha\cap M$. Then we set
      $ G^M(\zeta)=
       \{\n q[G^M_\zeta]:\, (\exists p)\, p\cup (\zeta,\n q)\in G^M_{\zeta+1}\}
      $.
      I.e., $G^M(\zeta)$ is the usual $Q_\zeta^{M[G^M_\zeta]}$-generic
      filter over  $M[G^M_{\zeta}]$ as defined in $M[G^M_\alpha]$.
    \item $q$ is $(M,P_\alpha)$-generic means that
      $q\in P_\alpha$ forces that the $P_\alpha$-generic filter
      $G$ is $(M,P_\alpha)$-generic.
      If $p\in P^M_\alpha$ (or if $p$ is just a $P^M_{\alpha}$-name
      (in $M$)
      for some $P^M_\alpha$-condition),
      then $q$ is $(M,P_\alpha,p)$-generic,
      if $q$ additionally forces that $p\in G^M_\alpha$
      (or that $p[G^M_\alpha]\in G^M_\alpha$, resp.).
  \end{itemize}
\end{Def}

The following is an immediate consequence of the definition:
\begin{Facts}
  \begin{itemize}
    \item If $\zeta\in M\cap \alpha$, then
      $G^M(\zeta)=Q_\zeta^{M[G^M_\zeta]}\cap G(\zeta)$.
    \item If $q$ is  $(M,P_\alpha,p)$-generic and 
      $\zeta\in M\cap \alpha$, then $q\restriction\zeta$
      is $(M,P_\zeta,p\restriction\zeta)$-generic.
  \end{itemize}
\end{Facts}

$G^M_{\alpha}$ is absolute in the following sense:
\begin{Lem}\label{fct:GMtrans}
  Assume that $M,N$ are candidates in $V$, $M\in N$, $V'$ is an extension
  of $V$, $\alpha\in M\cap \epsilon$, and
  $G\subset P_\alpha$ is an element of $V'$ which is
  $(N,P_\alpha)$-generic.
  \begin{enumerate}
    \item $G^{M}_\alpha$ (in $V'$) is the same as
       $G^{N}_\alpha)^{M}_\alpha$ (in $N[G^N_\alpha]$). 
       In other words,
       the $P^M_\alpha$-filter
       calculated in $V'$ from $G$
       is the same as the $P^M_\alpha$-filter
       calculated in $N[G^N_\alpha]$ from $G^N_\alpha$.
    \item In particular,
      $G_\alpha$ is $(M,P_\alpha)$-generic iff $N$ thinks that 
      $G^{N}_\alpha$ is $(M,P_\alpha)$-generic.
    \item If $G$ is $(M,P_\alpha)$-generic and $\tau$ a $P^M_\alpha$-name
      (in $M$), then
      ``$x=\tau[G^{M}_\alpha]$'' is absolute between $N[G^{N}_\alpha]$ and
      $V'$.
  \end{enumerate}
\end{Lem}

\begin{proof}
  By induction on  $\alpha\in \epsilon\cap M$: (2) follows from (1)
  by definition, and (3) from (1) using~\ref{lem:evalabs}.

  Assume $\alpha=\zeta+1$. Then $p\in G^{M}_\alpha$ iff 
  $p\restriction\zeta\in  G^{M}_\zeta$ and 
  $p(\zeta)[G^{M}_\zeta]\in G(\zeta)$ 
  iff $N[G^N_\zeta]\vDash p\restriction\zeta\in  G^{M}_\zeta$
  (by induction hypothesis 1)
  and $N[G^N_\zeta]\vDash p(\zeta)[G^{M}_\zeta]\in G^N(\zeta)$
  (by induction hypothesis 3
  and the fact that $M[G^{M}_\zeta]\vDash q\in Q_\zeta$ implies
  $N[G^N_\zeta]\vDash q\in Q_\zeta$).

  Now assume $\alpha$ is a limit. Then $p\in G^{M}_\alpha$ iff ($p\restriction\zeta \in
  G^{M}_\zeta$ for all $\zeta$) iff ($N\vDash p\restriction\zeta \in G^{M}_\zeta$ for all $\zeta$)  (by induction hypothesis 1) iff
  $N\vDash p\in G^{M}_\alpha$.
\end{proof}

So here we use that $Q_\alpha$ is upwards absolutely defined in $V[G_\alpha]$,
and that $M[G^{M}_\zeta]\in N[G^{N}_\zeta]$ both are candidates.

The definitions are compatible with ord-collapses of elementary submodels:
\begin{Lem}\label{lem:pmcompwithesm}
  Let $N\esm H(\chi)$, $M=\ordcol(N)$, and let
  $G$ be $P_\alpha$-generic over $V$. Then
  \begin{enumerate}
    \item
      $G$ is $N$-generic iff it is $(M,P_\alpha)$-generic.
  \setcounter{tmp}{\value{enumi}}	
  \end{enumerate}
  If $G$ is $N$-generic and $p,\n\tau\in N$, then 
  \begin{enumerate}
  \setcounter{enumi}{\value{tmp}} 
    \item
      $p\in G$ iff $\ordcol^{N}(p)\in G^{M}_\alpha$,
    \item
      $\ordcol^{N[G_\alpha]}(\n\tau[G])=(\ordcol^{N}(\n\tau))[G^{M}_\alpha]$;
    \item in particular, $(M[G^{M}_\alpha], M,\in)$
      is the ord-collapse of $(N[G_\alpha],N,\in)$.
  \end{enumerate}
\end{Lem}
\begin{proof}
  The image of $x$ under the ord-collapse (of the appropriate model, i.e.,
  either $N$ or $N[G]$) is denoted by $x'$. Induction on $\alpha$:

  (1,2 successor, $\alpha=\zeta+1$.)
  Assume that $G\restriction P_\zeta =:G_\zeta$ is $P_\zeta$-generic over $N$.
  Fix $p\in P_\alpha\cap N$. Then
  $p\in G$ iff
  $\{\ p\restriction \zeta\in G_\zeta$ and $p(\zeta)[G_\zeta]\in G(\zeta)\ \}$
  iff $\{ p'\restriction \zeta\in G^{M}_\zeta$ (according to induction
  hypothesis (2))
  and $p'(\zeta)[G^{M}_\zeta]\in G(\zeta)$ (according to 
  induction hypothesis (3)) $\}$%
  \footnote{using the fact that
    that $p(\zeta)[G_\zeta]\in Q_\zeta$ is hereditarily countable
    modulo ordinals and therefore not changed by the collapse} 
  iff $p'\in G^{M}_\alpha$.

  (1,2 limit)
  Assume that $G_\zeta$ is $P_\zeta$-generic over $N$ for all $\zeta\in \alpha\cap N$.
  Fix $p\in P_\alpha\cap N$.
  $p\in G$ iff
  $\{\ p\restriction \zeta\in G_\zeta$ for all $\zeta\in\alpha\cap N\ \}$ iff
  $p'\in G^{M}_\zeta$ for all $\zeta\in\alpha\cap M\ \}$ (by hypothesis (2))
  iff $p'\in G^M_\alpha$.

  (3) Induction on the depth of the name $\n\tau$:
  \\
  Let $A\in N$ be a  maximal antichain deciding whether $\n\tau \in V$
  (and if so, also the value of $\n\tau$).
  Assume $a\in A\cap G\cap N$. If $a$ forces $\n\tau=\std x$ for $x\in
  V$,
  then $M$ thinks that $a'\in G^{M}_\alpha$ forces $\n\tau'$ to be $x'\in V$,
  so we get $\n\tau'[G^{M}_\alpha]=x'$.
  If $a$ forces that $\n\tau\notin V$, then
  \[
    \n\tau'[G^{M}_\alpha]=\{\n\sigma'[G^{M}_\alpha]:\, (\sigma,p)\in \tau\cap N, p'\in G^{M}_\alpha\}
  =\{(\n\sigma[G])':\, (\sigma,p)\in \tau\cap N, p\in G\}
  \]
  (by induction).
  It remains to be shown that this is the ord-collapse of
  $\n\tau[G]=\{\n\sigma[G]:\, \sigma\in \tau\}$. For this it is
  enough to note 
  that for all $\rho[G]\in \n \tau[G]\cap N[G]$ 
  there is a $(\n \sigma,p)\in\n\tau$ such that $p\in G$ and
  $\n\sigma[G]=\rho[G]$.
\end{proof}

$P_\alpha$ satisfies (a version of) the properness condition for 
candidates:

\begin{Lem}\label{lem:csiisproper}
  For every candidate $M$ and $p\in P^M_\alpha$ there
  is an $(M,P_\alpha,p)$-generic $q$ such that 
  $\dom(q)\subseteq M\cap \alpha$.
\end{Lem}

The proof is more or less the same as the iterability of properness given in
\cite{MR1234283}. Since we will later need a ``canonical'' version of the
proof, we will introduce the following notation:

\begin{Def} For $\alpha<\epsilon$, let
  $\gen_\alpha$ be a $P_\alpha$-name for a
  function such that the following is forced by $P_\alpha$:
  If $M$ is a candidate, $\sigma:\omega\to M$ surjective, and
  $p\in Q^M_\alpha$ then $\gen_\alpha(M,\sigma,p)$
  is an $M$-generic element of $Q^{V[G_\alpha]}$
  stronger than $p$.
\end{Def}

(It is clear that such functions exist, since we assume that 
$P_\alpha$ forces that $Q_\alpha$ is nep. Later we will assume
that we can pick $\gen_\alpha$ in some 
absolute way, cf. \ref{asm:absolgen}).

For $\alpha\leq \beta<\epsilon$, let $P_{\beta/\alpha}$ denote the set of
$P_\alpha$-names $p$ for elements of $P_\beta$ such that $P_\alpha$ forces
$p\restriction\alpha\in G_\alpha$.  (I.e., $P_{\beta/\alpha}$ is 
a $P_\alpha$-name for the quotient forcing.) As usual,
we can define the $M$-version: $p\in P_{\beta/\alpha}^M$
means that $p$ is a $P^M_\alpha$-name (in $M$) for a $P^M_\beta$-condition,
and if $G$ is $(M,P_\alpha)$-generic, then
$p[G^M_\alpha]\restriction \alpha\in G^M_\alpha$.

Lemma \ref{lem:csiisproper} is a special case of the
following:
\begin{inductionlemma}\label{ind:gurke}
  Assume that $M$ is a candidate, $\sigma:\omega\to M$ surjective,
  $\alpha,\beta\in M$, $\alpha\leq \beta\leq \epsilon$, 
  $p\in P_{\beta/\alpha}^M$, $q$ is $(M,P_\alpha)$-generic,
  and that $\dom(q)\subseteq \alpha\cap M$. We define the canonical 
  $(M,\sigma,P_\beta,p)$-extension $q^+$ of $q$ such that
  $q^+\in P_\beta$ and $q^+$ is $(M,P_\beta,p)$-generic and
  $\dom(q^+)\subseteq M\cap\beta$.
\end{inductionlemma}

\begin{proof}
  Induction on $\beta\in M$.

  Successor step $\beta=\zeta+1$: By induction we have the
  canonical $(M,\sigma,P_\zeta,p\restriction \zeta)$-extension 
  $q^+\in P_\zeta$. In particular,
  $q^+$ forces that $M'\DEFEQ M[G^M_\zeta]$ is a candidate and that
  $p'\DEFEQ p(\zeta)[G^M_\zeta]\in Q_\zeta^{M'}$. 
  By applying $\sigma$ to the $P^M_\zeta$-names in $M$, we get a canonical
  surjection $\sigma':\omega\to M'$.  We define the canonical
  $\beta$-extension $q^{++}$ to be $q^+\cup{(\zeta,\gen_\zeta(M',\sigma',p'))}$.
  Assume that $G_\beta$ contains $q^{++}$. 
  Then  $G^M_\zeta$ is $P^M_\zeta$-generic
  and contains $p\restriction\zeta$. If $A\subseteq P^M_\beta$ is (in $M$)
  a maximal antichain, then
  \[
    A'\DEFEQ\{a(\zeta):\, a\in A,\, a\restriction\zeta
    \in G^M_\zeta\}\subseteq Q_\zeta^{M[G^M_\zeta]}
  \]
  is a maximal antichain
  in $M[G^M_\zeta]$. Since $\gen_\zeta(M',p')$ is in $G(\zeta)$, there
  is exactly one $a'\in A'\cap G(\zeta)$, i.e., there is exactly one
  $a\in A\cap G^M_\beta$. So $q^{++}$ is really $(M,P_\beta,p)$-generic.
  
  Limit step: Assume $\alpha=\alpha_0<\alpha_1\dots$ is cofinal in $M\cap
  \beta$.
  Set
  $D_0=P^M_\beta$ and let  $(D_n)_{n\in\omega}$ enumerate the $P^M_\beta$-dense
  subset in $M$. (Note that we get this enumeration canonically from $\sigma$.)
  First we define $(p_n)_{n\in\omega}$ such that $p_0=p$, $p_n\in P^M_{\beta/\alpha_n}$,
  and ($M$ thinks that) $P^M_{\alpha_n}$ forces
  \begin{itemize}
    \item $p_n\in D_n$,
    \item $p_{n-1}\restriction \alpha_n\in
      G^M_{\alpha_n}$ implies $p_n\leq p_{n-1}$.
  \end{itemize}
  Then we pick $q=q_{-1}\subseteq q_0 \subseteq q_1\dots$ such that $q_{n}$
  is the canonical $(M,\sigma,P_{\alpha_{n+1}},p_n\restriction \alpha_{n+1})$-generic
  extension of the $(M,P_{\alpha_{n}})$-generic $q_n$.
  We set $q^+\DEFEQ \bigcup_{n\in\omega} q_n$.

  By induction we get:
  \begin{itemize}
    \item $q_n$ is $(M,P_{\alpha_{n+1}},p_m\restriction\alpha_{n+1})$-generic
      for all $m\leq n$.
    \item $q_n$ forces 
      $p_l[G^M_{\alpha_{n+1}}]\leq p_m[G^M_{\alpha_{n+1}}]$
      (in $P^M_\beta$) for $m\leq l\leq n+1$.
  \end{itemize}
  $q^+$ is $G^M_\beta$ is $(M,P_\beta,p)$-generic: Let $G$ be $V$-generic and
  contain $q^+$.
  \begin{itemize} 
     \item $G^M_\beta$ meets $D_n$: $p_n[G^M_\beta]\in G^M_\beta$,
       since $p_n[G^M_\beta]\restriction \alpha_{m+1}\in G^M_\beta $ for all
       $m\geq n$.
     \item Let $r,s$ be incompatible in
       $P^M_\beta$. 
       In $M$, the set 
       \[
         D=\{p\in P^M_\beta:\, (\exists \zeta<\beta)(\exists t\in\{r,s\})
         p\restriction\zeta\forc_\zeta t\notin G_\zeta\}\subseteq P^M_\beta
       \]
       is dense, and if
       $p\in D\cap G^M_\beta$, then $p\restriction\zeta\in G^M_\zeta$,
       so $t\restriction\zeta\notin G^M_\zeta$, and $t\notin G^M_\beta$. 
  \end{itemize}
\end{proof}

We repeat Lemma \ref{lem:csiisproper} with our new notation:
\begin{Cor}
  Given a candidate $M$, $\sigma:\omega\to M$ surjective,
  and $p\in P^M_\alpha$, we can define
  the canonical generic $q=\gen(M,\sigma,P_\alpha,p)$.
  Also, $\dom(q)\subseteq M\cap \alpha$.
\end{Cor}

So $P_\alpha$ satisfies the properness-clause of the nep-definition.
However, $P_\alpha$ is not nep, since the statement ``$p\in P_\alpha$'' is not
upwards absolute.

\begin{Rem}
  There are two obvious reasons why ``$p\in P_\alpha$'' is not upwards absolute:
  First of all, names look entirely different in various
  candidates. For example, if $M$ thinks that 
  $\n\tau$ is the standard name for $\om1$, then a bigger candidate
  $N$ will generally see that $\n\tau$ is not a standard name for $\om1$.
  So if $Q$ is the (trivial) forcing $\{\om1\}$, then a condition
  in $P\ast Q$ is a pair $(p,q)$, where $P$ is (essentially) a standard
  $P$-name for $\om1$. So if $M$ thinks that $(p,q)\in P\ast Q$, then
  $N$ (or $V$) will generally not think that $(p,q)\in P\ast Q$.
  So we cannot use the formula ``$p\in P_\alpha$'' directly. We will use 
  pairs $(M,p)$ instead, where $M$ is a candidate and $p\in P^M_\alpha$.
  Another way to circumvent this problem would be
  to use absolute names for $\hco$-objects
  (inductively defined, starting with, e.g., ``standard name for $\alpha$'',
  and allowing ``name for union of $(x_i)_{i\in\omega}$'' etc).
  The second reason is that forcing
  is generally not absolute (even when we use absolute names): 
  $M$ can wrongly think that $p$ forces that $\n q_2\leq \n q_1$,
  i.e., that $(p,\n q_2)\leq (p,\n q_1)$ in $P\ast \n Q$.
  We will avoid this by interpreting 
  $(M,p)$ to be a {\em canonical} $(M,P_\alpha,p)$-generic.
\end{Rem}

\subsection{The nep iteration: properness and absoluteness.}\label{ss:absolute}

Now we will construct a version of $P_\epsilon$ that is forcing equivalent to
the usual countable support iteration and upwards absolutely defined.
We will need to construct generics in a canonical  way, so we assume the
following:
 
\begin{Asm}\label{asm:absolgen}
    There is an absolute (definition of a) 
    function $\gen_\alpha$
    such that $P_\alpha$ forces:
    If $M$ is a candidate, 
    $\sigma: \omega\to M$ surjective
    and
    $p\in Q^ M_\alpha$
    then $\gen_\alpha(M,\sigma ,p)$ is $Q^M_\alpha$-generic
    over $M$ and stronger than $p$.
\end{Asm}

\begin{Def}
  $P^\textrm{nep}_\alpha$ consists of tuples $(M,\sigma,p)$, where $M$ is
  a candidate, $\sigma:\omega\to M$  surjective, and $p\in P^M_\alpha$.
\end{Def}

So ``$x\in P_\alpha^\textrm{nep}$'' obviously is upwards absolute.

We will interpret $(M,\sigma,p)$ as 
the ``canonical $M$-generic condition forcing that $p\in G^M_\alpha$''.
(Generally there are many generic conditions, and incompatible ones, so we have
to single out a specific one, the canonical generic, and for this we need
\ref{asm:absolgen}).

Recall the construction of $\gen$ from Definition/Lemma~\ref{ind:gurke}.
If we use Assumption~\ref{asm:absolgen}, we get:
\begin{Cor}\label{lem:genprops}
  $\gen: P^\textrm{nep}_\alpha\to P_\alpha$ is such that
  \begin{enumerate}
    \item $\gen(M,\sigma,p)$ is $(M,P_\alpha,p)$-generic
    \item If $M,N$ are candidates, $M,\sigma\in N$, and $G$ is
      $(N,P_\alpha)$-generic, then
      $\gen^{N}(M,\sigma,p)\in G^{N}_\alpha$ iff $\gen^{V}(M,\sigma,p)\in G$.
  \end{enumerate}
\end{Cor}
(Here, $\gen^{N}$ is the result of the construction~\ref{ind:gurke}
carried out inside $N$, and analogously for $V$.)
Of course $\gen$ cannot really be upwards absolute (i.e., we cannot have
$\gen^{N}(M,\sigma,p)=\gen^{V}(M,\sigma,p)$), since $x\in P_\alpha$ is not
upwards absolute.  However, (2) gives us a sufficient amount of absoluteness.

\begin{proof}
  (1) is clear. For (2), just go through the construction of~\ref{ind:gurke}
  again and check by induction 
  that this construction is really sufficiently ``canonical'', i.e.,
  absolute.
\end{proof}

If $\sigma_1\neq \sigma_2$ both enumerate $M$, then we do not require
$\gen(M,\sigma_1,p)$ and $\gen(M,\sigma_2,p)$ to be compatible.

Let us first note that a function $\gen$ as above also satisfies the following:

\begin{Cor}\label{cor:gencor}
  \begin{enumerate}
    \item If $N$ thinks that $q\leq^*\gen^{N}(M,\sigma_M,p)$,
      then $\gen^V(N,\sigma_N,q)\leq^*\gen^V(M,\sigma_M,p)$.
    \item If $N\esm H(\chi)$, $p\in N$, and $(N',p')$ is the ord-collapse
      of $(N,p)$, and $\sigma': \omega\to N'$ is surjective,
      then $\gen(N',\sigma',p')\leq^* p$.
  \end{enumerate}
\end{Cor}

\begin{proof}
  (1)
  Assume $G_\alpha$ contains $\gen^V(N,\sigma_N,q)$. So
  $G_\alpha$ is $(N,P_\alpha)$-generic and $G^N_\alpha$ contains
  $q$ and therefore
  $\gen^N(M,\sigma_M,p)$. So by \ref{lem:genprops}(2),
  $G_\alpha$ contains $\gen^V(M,\sigma_M,p)$.

  (2)
  Assume that the $V$-generic filter $G$ contains $\gen(N',\sigma',p')$.
  Then by definition, $G^{N'}$ is $N'$-generic and contains $p'$.
  So $G$ is $N$-generic and contains $p$,
  according to~\ref{lem:pmcompwithesm}.
\end{proof}

Now we can define: 
\begin{Def}
  $(M_2,\sigma_2,p_2)\leq^\textrm{nep} (M_1,\sigma_1,p_1)$ means: 
  $M_1,\sigma_1\in M_2$, and $M_2$ thinks that 
  ($M_1$ is a candidate and that)
  $p_2\leq^* \gen^{M_2}(M_1,\sigma_1,p_1)$.
\end{Def}

By Corollary~\ref{cor:gencor}(1) $\leq^\textrm{nep}$ is transitive. 
It follows:
\begin{Thm}
  \begin{enumerate}
    \item $\gen: (P^\textrm{nep}_\alpha,\leq^\textrm{nep})\to (P_\alpha,\leq^*)$
      is a dense embedding.
    \item $(P^\textrm{nep}_\alpha,\leq^\textrm{nep})$ is nep.
  \end{enumerate}
\end{Thm}

\begin{proof}
  (1)
  \begin{itemize}
    \item If $(M_2,\sigma_2,p_2)\leq^\textrm{nep} (M_1,\sigma_1,p_1)$, then
      $\gen(M_2,\sigma_2,p_2)\leq^*\gen(M_1,\sigma_1,p_1)$
      by~\ref{cor:gencor}(1).
    \item If $(M_2,\sigma_2,p_2)\incomp^\textrm{nep} (M_1,\sigma_1,p_1)$, then
      $\gen(M_2,\sigma_2,p_2)\incomp^*\gen(M_1,\sigma_1,p_1)$:\\
      Assume that $q\leq^* \gen(M_i,\sigma_i,p_i)$ ($i\in\{1,2\}$).
      Let $N\esm H(\chi)$ contain 
      $q$ and $M_i,\sigma_i,p_i$, and let
      $(M_3,p_3)$ be the ord-collapse of $(N,q)$ and $\sigma_3:\omega\to M_3$
      surjective.
      Then $(M_3,\sigma_3,p_3)\leq^\textrm{nep}(M_i,\sigma_i,p_i)$.
    \item $\gen$ is dense:
      For $p\in P_\alpha$ pick an $N\esm H(\chi)$ containing $p$ and
      let $(N',p')$ be the ord-collapse of $(N,p)$. Then
      $\gen(N',p')\leq^* p$, according to  \ref{cor:gencor}.2.
  \end{itemize}

  (2): The definitions of $P^\textrm{nep}$ and $\leq^\textrm{nep}$ are
  clearly upwards absolute. If $N$ is a candidate and $N$
  thinks $(M,p)\in P^\textrm{nep}$, then 
  $(N,\gen^{N}(M,p))\leq^\textrm{nep}(M,p)$ is $N$-generic:

  Assume $G^\textrm{nep}$ is a $P^\textrm{nep}$-generic filter
  over $V$ containing $(N,q)$ (for some $q$). Since
  $\gen$ is a dense embedding, $G^\textrm{nep}$ defines a 
  $P_\alpha$-generic filter $G_\alpha$ over $V$, and $G_\alpha$
  contains $\gen(N,q)$. This implies 
  that $G_\alpha$ is $(N,P_\alpha)$-generic.

  We have to show that 
  $G^\textrm{neq}\cap P^{\textrm{nep},N}$ is $P^{\textrm{nep},N}$-generic over $N$.
  In $N$, the mapping $\gen^N: P^{\textrm{nep},N}\to P^N_\alpha$ is dense,
  and $G^N_\alpha$ is $P^N_\alpha$-generic over $N$. So the 
  set $G'=\{(M,p):\, \gen^N(M,p)\in G^N_\alpha \}$ is  
  $P^{\textrm{nep},N}$-generic over $N$. But $(M,p)\in G'$ iff
  $(M,p)\in G$, according to \ref{lem:genprops}(2).
\end{proof}

\begin{Rem}
  So the properties \ref{lem:genprops}(1) and \ref{cor:gencor}(1) are enough
  to show that $P^\textrm{nep}$ can be densely embedded into $P_\alpha$.
  But \ref{lem:genprops}(2) is needed to show that $P^\textrm{nep}$ actually
  is nep: Otherwise $P^\textrm{nep}$ still is an upwards absolute forcing 
  definition, and for every $p\in (P^\textrm{nep})^M$ there is a $q\leq p$
  in $P^\textrm{nep}$ forcing that there is an $(P^\textrm{nep})^M$-generic
  filter over $M$, namely the reverse image of $G^M_\alpha$ under
  $\gen^M$
  but this filter doesn't have to be the same as $G^\textrm{nep}\cap
  (P^\textrm{nep})^M$.
\end{Rem}

\subsection{Iterations along subsets of $\epsilon$}\label{ss:subsetiteration}

As before we assume that $(Q_\alpha)_{\alpha\in\epsilon}$ is 
a sequence of forcing-definitions.

We can of course define a countable support iteration along
every subset $w$ of $\epsilon$:

$P_w$, the c.s.-iteration of
$(Q_\alpha)_{\alpha\in w}$ along $w$, is defined by induction on $\alpha\in w$:
$P_{w\cap \alpha}$ consists of functions $p$ with countable domain $\subseteq
w\cap \alpha$.  If $\alpha$ is the $w$-successor of $\zeta$, then $p\in
P_{w\cap\alpha}$ iff $p\restriction\zeta\in P_{w\cap\zeta}$ and
$\zeta\notin\dom(p)$ or $p(\zeta)$ is a $P_{w\cap \zeta}$-name for for an
object in $Q_\alpha$.  If $\alpha$ is a $w$-limit, then $p\in
P_{w\cap\alpha}$ iff $p\restriction\zeta\in P_{w\cap\zeta}$ for all
$\zeta\in\alpha\cap w$.

Of course this notion does not bring anything new: Assume $\beta\leq\epsilon$
is the order type of $w$, and let $i: \beta\to w$ be the isomorphism.
Then $P_w$ is isomorphic to the c.s.-iteration
$(R_\alpha,Q_{i(\alpha)})_{\alpha<\beta}$.

We can calculate $P_w$ inside $M$ and extend our notation to that case:

\begin{Def}\label{def:iterationalongsubsets}
   Let $M$ be a candidate, $w\subseteq \epsilon$, $w\in M$.
  \begin{itemize}
    \item $P_w$ is the 
      countable support iteration along the order $w$.
    \item $P^M_w$ is the forcing $P_w$ as constructed in $M$.
    \item $v$ covers $w$ (with respect to
      $M$) if $\epsilon\supseteq v \supseteq w\cap M$.\\
      (If $w\notin\ON$, then this is independent of $M$, since $w\subseteq M$
      for each candidate $M$.)
    \item If $v$ covers $w$, and $G_v\subseteq P_v$, then we define $G^M_{v\to
      w}$ by the following 
      induction on $\alpha\in\epsilon\cap M$, provided this
      results in a $P^M_w$-generic
      filter over $M$. Otherwise, $G^M_{v\to w}$ is undefined.
      Let $p\in P^M_{\cap w\cap \alpha}$.
      If $M$ thinks that $w\cap \alpha$ has no last element, then
      $p\in G^M_{v\to w\cap \alpha}$ iff $p\restriction \beta\in G^M_{v\to w\cap \beta}$
      for all $\beta\in \alpha\cap M$. If $w\cap \alpha$ has the last element $\beta$,
      then $p\in G^M_{v\to w\cap \alpha}$ iff $p\restriction \beta\in G^M_{v\to w\cap \beta}$
      and $p(\beta)[G^M_{v\to w\cap \beta}]\in G_v(\beta)$.
    \item A $G_v$ such that $G^M_{v\to w}$ is defined is called
      $(M,P_w)$-generic.
    \item Assume that $G_v$ is  $(M,P_w)$-generic and
      $\zeta\in w\cap M$. Then we set
      \[ G^M_{v\to w}(\zeta)=
       \{\n q[G^M_{v\to w\restriction\zeta}]:\, \exists p\, p\cup(\zeta,\n q)\in G^M_{v\to w}\}.
      \]
    \item If $v$ covers $w$, $q\in P_v$, and 
      $p$ is a $P^M_w$-condition (or $p$ is just in $M$ a 
      $P^M_w$-name for a $P^M_w$-condition),
      then $q$ is $(M,P_{v\to w},p)$-generic if $q$ forces
      that $G_v$ is $(M,P_w)$-generic and $p\in
      G^M_{v\to w}$ (or $p[G^M_{v\to w}]\in G^M_{v\to w}$,
      resp.).
  \end{itemize}
\end{Def}

The same proofs as~\ref{fct:GMtrans},~\ref{lem:pmcompwithesm} and~\ref{lem:csiisproper} give us the according results for $P_w$:
\begin{Lem}\label{lem:jhet}
  Assume that $V'$ is an extension of $V$
  \begin{itemize}
    \item $M$ and $N$ are candidates in $V$, $M\in N$,
    \item $v\in N$ and $u$ covers $v$ with respect to $N$,
    \item $w\in M$, $M\in N$ and $N$ thinks that $M$ is a candidate and that 
       $v$ covers $w$ with respect to $M$,
    \item $G_u\in V'$ is $(N,P_v)$-generic.
  \end{itemize}
  Then we get:
  \begin{enumerate}
    \item In $V$, $v$ covers $w$ with respect to $M$.
    \item If $\zeta\in v\cap N$, then
      $G^N_{u\to v}(\zeta)=Q_\zeta^{N[G^N_{u\to v}]}\cap G_u(\zeta)$.
    \item $(G^M_{u\to w})^{V'}=(G^M_{v\to w})^{N[G^N_{u\to v}]}$.
    \item In particular, $G$ is $(M,P_{u\to w})$-generic iff $N[G^N_{u\to v}]$ thinks that
      $G^N_{u\to v}$ is $(M,P_{v\to w})$-generic.
    \item If $G_u$ is $(M,P_w)$-generic and $\n\tau$ a $P^M_w$-name in $M$, then
      $\n\tau[G^M_{u\to w}]$ (calculated in $V[G_u]$)
      is the same as $\n\tau[G^M_{v\to w}]^M$ 
      (calculated in $N[G^N_{u\to v}]$).
    \item If $q$ is $(N,P_{u\to v},p)$-generic and $\alpha\in \epsilon\cap N$,
      then $q\restriction\alpha$ is $(N,P_{u\restriction\alpha \to
      v\cap\alpha},p\restriction \alpha)$-generic.
  \end{enumerate}
\end{Lem}


\begin{Lem}\label{lem:wilejt}
  Let $N\esm H(\chi)$, $v\in N$, 
  $(M,w)=\ordcol(N,v)$,\footnote{so either $w=v\cap N$ or $w=v\in\ON$} 
  and let $G_v$ be $P_v$-generic over $V$. Then 
  \begin{enumerate}
    \item
      $G_v$ is $N$-generic iff it is $(M,P_{v\to w})$-generic.
  \setcounter{tmp}{\value{enumi}}	
  \end{enumerate}
  If $G_v$ is $N$-generic and $p,\n\tau\in N$, then 
  \begin{enumerate}
  \setcounter{enumi}{\value{tmp}} 
    \item
      $p\in G_v$ iff $\ordcol^{N}(p)\in G^{M}_{v\to w}$, and
    \item
      $\ordcol^{N[G_v]}(\n\tau[G_v])=(\ordcol^{N}(\n\tau))[G^{M}_{v\to w}]$.
  \end{enumerate}
\end{Lem}

\begin{Lem}
  If $M$ is a candidate, $w\in M$, $v$ covers $w$, and 
  $p\in P^M_w$, then there is a $(M,P_{v\to w},p)$-generic $q\in P_v$. 
\end{Lem}

We give the following Lemma (used for $Q=$Mathias in \cite{MR2353856})
as an example for how we can use this iteration:
\begin{Lem}\label{lem:iterationapplication}
  Let $\bar B=(B_i)_{i\in I}$ be a sequence (in $V$) of Borel codes.
  Let $Q_\alpha=Q$ be the same
  nep forcing (definition) for
  all $\alpha<\epsilon$. 
  If $P_{\om1}$ forces $\bigcap\bar B=\emptyset$,
  then $P_{\epsilon}$ forces $\bigcap \bar B=\emptyset$.
\end{Lem}

\begin{proof}
  We assume that
  $\bigcap\bar B=\emptyset$ is forced by $P_{\om1}$ and 
  therefore by all $P_{\alpha}$ for $\alpha\in\om1$.
  We additionally assume towards a contradiction that
  \begin{equation}\label{eq:weth23}
    p_0\forc_{\epsilon} \n \eta_0\in \bigcap B_i.
  \end{equation}
  We fix a ``countable version'' of the name $\n \eta_0$:
  Let $N_0\esm H(\chi)$ contain $\n\eta_0$ and $p_0$.
  Let $(M_0,\n\eta_0',p'_0)$ 
  be the ord-collapse of $(N_0,\n\eta_0,p_0)$.
  Set $w=\epsilon\cap N_0=\epsilon\cap M_0$.
  In particular, $w$ is countable.

  Since $w$ covers $\epsilon$ with respect to $M_0$, we can find
  an $(M_0,P_{w\to \epsilon},p'_0)$-generic
  condition $q_0$ in $P_w$.
  Under $q_0$ we can
  define the $P_{w}$-name 
  \begin{equation}\label{eq:weth24}
    \n\tau\DEFEQ \n\eta'_0[G^{M_0}_{w\to\epsilon}].
  \end{equation}

  So whenever $q$ is in a $G_w$-generic filter,
  then $\n\tau[G_w]$ is the same as $\n\eta'_0[G^{M_0}_{w\to\epsilon}]$.

  $P_{w}$ is isomorphic to $P_\alpha$
  for some countable $\alpha$, so
  we know that $P_{w}$ forces $\n\tau\notin \bigcap \bar B$.
  In particular, we can find a $\tilde q\leq q_0$ and an $i_0\in I$
  such that 
  \begin{equation}\label{eq:weth25}
    \tilde q\forc \n\tau\notin B_{i_0}.
  \end{equation}

  Let $N_1\esm H(\chi)$ contain the previously mentioned objects.
  In particular $w\subset N_1$.
  Let $(M_1,P',\tilde q')$ be the ord-collapses of $(N_1,P_w,\tilde q)$.
  By elementarity, $P'=P^{M_1}_w$.
  Since $\epsilon$ covers $w$, we can find 
  an $(M_1,P_{\epsilon\to w},q'_0)$-generic condition $q_1$ in
  $P_{\epsilon}$.

  Let $G$ be a $P_{\epsilon}$-generic filter over $V$ 
  containing $q_1$. Set $r=\n\eta_0[G]$. 
  So $r\in\bigcap \bar B$ by~\eqref{eq:weth23}.
  On the other hand, $\tilde r\DEFEQ \n\tau'[\tilde G]$ 
  is not in $B_{i_0}$ for
  $\tilde G\DEFEQ G^{M_1}_{\epsilon\to w}$.
  Also, $\tilde r=\n\eta'_0[\tilde G^{M_0}_{w\to\epsilon}]$.
  It remains to show that $r=\tilde r$.
  This follows from transitivity (see Lemma~\ref{lem:jhet}), i.e.,
  $\tilde G^{M_0}_{w\to\epsilon}= G^{M_0}_{\epsilon\to\epsilon}$,
  and the fact that $G^{M_0}_{\epsilon\to\epsilon}=G^{M_0}_\epsilon$,
  and from elementarity (see Lemma~\ref{lem:pmcompwithesm}), i.e.,
  $(M_0[G^{M_0}_\epsilon],M_0)$ is the
  ord-collapse of $(N_0[G],N_0)$.
\end{proof}

\bibliographystyle{amsalpha}
\bibliography{nep}

\end{document}